\documentclass[11pt]{amsart}

\usepackage{amssymb,amscd,amsthm,amsxtra}
\usepackage{dsfont,latexsym, color}
\usepackage{mathrsfs}
\usepackage{marvosym}
\usepackage{color}
\definecolor{citation}{rgb}{0.2,0.58,0.2} 
\definecolor{formula}{rgb}{0.1,0.2,0.6}
\definecolor{url}{rgb}{0.3,0,0.5} 
\definecolor{marrone}{rgb}{0.7,0.45,0.36} 
\usepackage[colorlinks=true,linkcolor=formula,urlcolor=url,citecolor=citation]{hyperref} 
\newtheorem{thm}{Theorem}[section]
\newtheorem{corollary}[thm]{Corollary}
\newtheorem{lemma}[thm]{Lemma}
\newtheorem{prop}[thm]{Proposition}
\theoremstyle{definition}

\theoremstyle{rem}
\newtheorem{rem}[thm]{Remark}
\numberwithin{equation}{section}
\newcommand{\R}{{\mathds R}}
\newcommand{\N}{{\mathds N}}
\newcommand{\Z}{{\mathds Z}}

\newcommand{\cc}{{\mathcal C}}

\newcommand{\ff}{{\mathcal F}}
\newcommand{\ii}{{\mathcal I}}

\newcommand{\Om}{\Omega}
\newcommand{\Omb}{\overline{\Omega}}
\newcommand{\bb}{\dot{B}^{-\alpha}_{\infty,\infty}}

\def\Xint#1{\mathchoice
      {\XXint\displaystyle\textstyle{#1}}%
      {\XXint\textstyle\scriptstyle{#1}}%
      {\XXint\scriptstyle\scriptscriptstyle{#1}}%
      {\XXint\scriptscriptstyle\scriptscriptstyle{#1}} %
\!\int}
   \def\XXint#1#2#3{{\setbox0=\hbox{$#1{#2#3}{\int}$}
        \vcenter{\hbox{$#2#3$}}\kern-.5\wd0}}
   
   \def\dashint{\Xint-}
\newcommand{\ue}{u_{\varepsilon}}

\newcommand{\tows}{\stackrel{\ast}{\rightharpoonup}}

\newcommand{\tow}{\rightharpoonup}

\def\dys{\displaystyle}

\def\eps{\varepsilon}

\def\mea{\mathcal{M}({\mathds{R}^N})}

\def\ue{u_{\eps}}

\newlength{\defbaselineskip}
\allowdisplaybreaks
\makeatletter
\DeclareRobustCommand*{\bfseries}{%
  \not@math@alphabet\bfseries\mathbf
  \fontseries\bfdefault\selectfont
  \boldmath
}
\makeatother

\title 
[Improved Sobolev embeddings]
{Improved Sobolev embeddings, profile decomposition,   
 and concentration-compactness  for fractional Sobolev spaces}
\author[G. Palatucci]{Giampiero Palatucci}
\email[G. Palatucci]{\href{mailto:giampiero.palatucci@unimes.fr}{giampiero.palatucci@unimes.fr}}
\author[A. Pisante]{Adriano Pisante}
\email[A. Pisante]{\href{mailto:pisante@mat.uniroma1.it}{pisante@mat.uniroma1.it}}

\address[A. Pisante]
{Dipartimento di Matematica,  Sapienza Universit\`a di Roma
\\ P.~\!le Aldo Moro, 5
\\ 00185 Roma, Italia}

\address[G. Palatucci]
{
Dipartimento di Matematica e Informatica, Universit\`a degli Studi di Parma
\\ Campus - Parco Area delle Scienze,~53/A
\\ 43124 Parma, Italia}

\begin{document}

\subjclass[2010]{Primary 35J60;
Secondary 35C20, 35B33, 49J45\vspace{1mm}}

\keywords{Refined Sobolev inequalities, concentration-compactness principle, profile decomposition, critical Sobolev exponent, dislocation spaces, Morrey spaces, Besov spaces, fractional Sobolev spaces.
\vspace{1mm}}

\begin{abstract}
We obtain an improved Sobolev inequality in $\dot{H}^s$ spaces involving Morrey norms.
This refinement yields a direct proof of the existence of optimizers and the compactness up to symmetry of optimizing sequences for the usual Sobolev embedding.
More generally, it allows to derive an alternative, more transparent proof of the profile decomposition in $\dot{H}^s$ obtained 
in~\cite{Ger98} 
using the abstract approach of dislocation spaces developed in \cite{TF07}.
We also analyze directly the local defect of compactness of the Sobolev embedding in terms of measures in the spirit of
~\cite{lions,lions2}. 
As a model application, we study the asymptotic limit of a family of subcritical problems, obtaining concentration results for the corresponding optimizers which are well known when $s$ is an integer (\cite{rey,Han91}, \cite{ChGe}).
\end{abstract}

\maketitle

\section{Introduction} 

Fractional Sobolev spaces have been a classical topic in Functional and Harmonic Analysis as well as in Partial Differential Equations all the time. A great attention has been focused on the study of problems involving fractional spaces, and, more recently, the corresponding nonlocal equations, both from a pure mathematical point of view and for concrete applications, since they naturally show up in many different contexts.
For an elementary introduction to this topic and a wide, but still very limited, list of related references we refer to~\cite{DPV12}.
Here we are interested in the simplest among them, namely  the spaces $H^s(\R^N)$ and $\dot{H}^s(\R^N)$, which are their homogeneous counterpart, and in the corresponding Sobolev inequalities.
\vspace{1mm}

Let $N\geq 1$ and for each $0<s<N/2$ denote by $H^s(\R^N)$ the usual\footnote{
We immediately refer to Section~\ref{sec_spaces} for the basic definitions and some properties of the relevant spaces we deal with in the paper.
} $L^2-$based fractional Sobolev spaces and $\dot{H}^s(\R^N)$ its homogeneous version
defined via Fourier transform as the completion of $C_0^\infty(\mathds{R}^N)$ under the norm 
\begin{equation}
\label{Hs0-norm}
\| u\|^2_{\dot{H}^s}=\int_{\mathds{R}^N}|\xi|^{2s}|\hat{u}(\xi)|^2 d\xi \, .
\end{equation}

In the present paper, we mainly focus our attention on fractional Sobolev embeddings 
$\dot{H}^s(\R^N) \hookrightarrow L^{2^\ast}\!(\R^N)$,
and the corresponding inequality
\begin{equation}\label{eq_sobolev}
\|u\|^{2^\ast}_{L^{2^\ast}}  \leq \, S^\ast \|u\|^{2^\ast}_{\dot{H}^s} \quad \forall u \in \dot{H}^s(\R^N),
\end{equation}
where 
$2^\ast=2N/(N-2s)$ is the critical Sobolev exponent, so that $2^\ast \in (2,\infty)$ as $s \in (0,\, N/2)$, which allows to define the {\em fractional Laplacian} e.~\!g. as a bounded linear operator $(-\Delta)^{s/2}: \dot{H}^s(\mathds{R}^N) \to L^2 (\mathds{R}^N)$.

\vspace{2mm}

First, we recall that, using the important results in \cite{lieb} on the related Hardy-Littlewood-Sobolev inequality, the optimal constant in the Sobolev inequality~\eqref{eq_sobolev} was computed in \cite[Theorem 1.1]{cotsiolis}, namely
\begin{equation}\label{def_sstar}
\dys
S^{\ast}=\left(2^{-2s}\pi^{-{s}}\frac{\mathbf\Gamma\left(\frac{N-2s}{2}\right)}{\mathbf\Gamma\left(\frac{N+2s}{2}\right)}\left[\frac{\mathbf\Gamma(N)}{\mathbf\Gamma(N/2)}\right]^{2s/N}\right)^{\!\!\frac{2^{\ast}}{2}}\!\!,
\end{equation}
 together with the explicit formula for those functions giving equality in the inequality. Precisely, for $u\neq 0$, we have equality in {\rm \eqref{eq_sobolev}} if and only if
\begin{equation}\label{def_talentiana}
\dys
u(x)=\frac{c}{(\lambda^2+|x-x_0|^2)^{\frac{N-2s}{2}}} \ \ \forall x\in \R^N,
\end{equation}
where $c \in \R\setminus\{0\}$ and $\lambda >0$ are constants and  $x_0\in \R^N$ is a fixed point.

When $s=1$ the Sobolev inequality \eqref{eq_sobolev} as well as the previous results have been proved in~\cite{talenti} and also in~\cite{aubin}, where the subtle connection of \eqref{eq_sobolev}-\eqref{def_talentiana} with the Yamabe problem in Riemannian geometry is discussed.
When $2\leq s<N/2$ is an even integer the same result was obtained some years later in~\cite{Sw}, following the ideas in~\cite{lions} and~\cite{lions2}.
Also, the case~$s=1/2$ has been already studied in~\cite{Es} in an equivalent form given by the $s$-harmonic extension (see \eqref{tracesobolev}-\eqref{fracextension} in Section~\ref{sec_fractional}), in connection with the Yamabe problem on manifolds with boundary.

\vspace{1mm}

Using the moving planes method, formula \eqref{def_talentiana} has been obtained independently by Chen, Li \& Ou in \cite{ouLi}. 
At least when $0<s<1$, a third approach through symmetrization techniques applied to the equivalent Gagliardo seminorm (see~\eqref{gagliardo} in Section~\ref{sec_fractional}) can be found in \cite{FS}.
\vspace{2mm}

A naive approach to the validity of \eqref{eq_sobolev} is to study the variational problem
\begin{eqnarray}\label{pbsobolev}
\dys S^{\ast} \! := \!\sup\left\{ F(u) :  u\in \dot{H}^s(\mathds{R}^N), \, \int_{\R^N} |(-\Delta)^{\frac{s}{2}} u|^{2} dx \leq 1 \right\}
\end{eqnarray}
\begin{eqnarray*}
\dys \text{where} \ \ F (u)\!\! := \!\!\int_{\mathds{R}^N} |u|^{2^{\ast}\!} dx. 
\end{eqnarray*}

Clearly, the validity of \eqref{eq_sobolev} is equivalent to show that the constant $S^{\ast}$ defined in~\eqref{pbsobolev} is finite, and an explicit formula of $S^\ast$ is given by~\eqref{def_sstar}. Moreover, \eqref{def_talentiana}~gives the maximizers of the variational problem \eqref{pbsobolev} up to normalization. Note that even the existence of a maximizer is not trivial since the embedding \eqref{eq_sobolev} is not compact, because of translation and dilation invariance. 
Indeed, if $u\in \dot{H}^s(\mathds{R}^N)$ is an admissible function in \eqref{pbsobolev}, the same holds for 
\begin{equation}\label{def_funzioni}
\dys 
u_{x_0,\lambda}(x)= \lambda^{-\frac{N-2s}{2}} u\left(\frac{ x-x_0}{\lambda} \right),
\end{equation}
for any $x_0\in \mathds{R}^N$ and any $\lambda>0$; in addition $u_{x_0,\lambda}$ satisfies $F(u_{x_0, \lambda})=F(u)$ and tends to zero weakly in $\dot{H}^s$, as $|x_0|\to \infty$ or as $\lambda\to 0^+$ and $\lambda\to \infty$. More precisely, formula \eqref{def_funzioni} and simple calculations show that both the $\dot{H}^s$ and the $L^{2^\ast}$ \!norms are invariant under translations and dilations and through \eqref{def_funzioni} such groups (actually their semidirect product) act by isometries on $\dot{H}^s$ and $L^{2^\ast}$ in a noncompact way.

Another related problem we consider is the following. Given a bounded domain $\Omega \subset \mathds{R}^N$, one can define the Sobolev space $\dot{H}^s(\Omega)$ as the closure of $C^\infty_0(\Omega)$ in $\dot{H}^s(\mathds{R}^N)$ with the norm in \eqref{Hs0-norm} and the related maximization problem (corresponding to Sobolev embedding $\dot{H}^s(\Omega) \hookrightarrow L^{2^\ast}(\Omega)$), namely
\begin{eqnarray}\label{pbsobolev2}
\dys S^{\ast}_\Omega \! := \!\sup\left\{ F_\Omega(u) : u\in \dot{H}^s(\Omega), \, \int_{\R^N} |(-\Delta)^{\frac{s}{2}} u|^{2} dx \leq 1 \right\}
\end{eqnarray}
\begin{eqnarray}\label{funzionale2}
\dys \text{where} \ \ F_\Omega (u)\!\! := \!\!\int_{\Omega} |u|^{2^{\ast}\!} dx.  
\end{eqnarray}

A simple scaling argument on compactly supported smooth functions shows that $S^{\ast}_\Omega=S^\ast$, but in view of~\eqref{def_talentiana}, the variational problem \eqref{pbsobolev2} has no maximizer and no maximizing sequence in \eqref{pbsobolev2} converges.

Part of our analyses here will be devoted to the study of the effects of such lack of compactness for optimizing sequences in \eqref{pbsobolev}-\eqref{pbsobolev2}. Also, in analogy with the the local case $s \in \mathds{N}$, we will describe such phenomena in terms of 
concentration-compactness and more generally in terms of the so-called profile decomposition as originally done in \cite{Ger98} for fractional spaces $ \dot{H}^s(\mathds{R}^N)$ but with a different approach.


In order to do this, among other tools we will need to make use of a refinement of the Sobolev inequality \eqref{eq_sobolev} itself, in terms of Morrey spaces. By a refinement of the Sobolev embedding, one means that there exists a Banach function space $X$ such that $\dot{H}^s\hookrightarrow X$ continuously (possibly $ \dot{H}^s\hookrightarrow L^{2^\ast} \hookrightarrow X$ ) and, for some $0< \theta < 1$ and some $C>0$,
\begin{equation}
\label{refsobolev}
\|u \|_{L^{2^\ast}} \, \leq \, C \|u\|^{\theta}_{\dot{H}^s} \|u\|_{X}^{1-\theta}, \quad \forall u\in \dot{H}^s(\R^N) \, .
\end{equation}
Clearly \eqref{refsobolev} implies \eqref{eq_sobolev} because the embedding $\dot{H}^s \hookrightarrow X$ is continuous.

Such improved Sobolev inequalities are difficult to get  but, as we will discuss below, they allow to obtain deeper informations that would be not detected in the Lebesgue scale.
The simplest choice is probably the Lorentz space $X=L(2^\ast,\infty) $, (so that $L^{2^\ast} \hookrightarrow X$), for which \eqref{refsobolev} can be proved combining Peetre's Sobolev embedding $\dot{H}^s \hookrightarrow L(2^\ast,2)$ with H\"older inequality in Lorentz spaces (see e.~\!g.~\cite{FS}) and for which the convexity exponent is $\theta=2/2^\ast$. 

In the same direction, Gerard, Meyer \& Oru (\cite{GMO97}) proved \eqref{refsobolev} when $X=  \dot{B}_{\infty,\infty}^{-N/2^\ast}$ is an homogeneous Besov space of negative smoothness and still $\theta=2/2^\ast$ and $L^{2^\ast} \hookrightarrow X$. Thus, in both cases the refinement \eqref{refsobolev} appears as an interpolation inequality between $\dot{H}^s$ and $X$ with intermediate space $L^{2^\ast}$. 
This is not the case in~\cite{Ger98}, pages~224-225, where using a result from \cite{GMO97} the inequality \eqref{refsobolev} is established for $X=\dot{B}^s_{2,\infty}$, a Besov space of positive smoothness very close to $\dot{H}^s$ (so that one only has $\dot{H}^s \hookrightarrow X \hookrightarrow L(2^\ast,\infty)$ but $X \not \hookrightarrow L^{2^\ast}$). We will back to this case below, when discussing the profile decomposition in $\dot{H}^s$.  

\vspace{2mm}
Now, for any $1\leq r<\infty$ and any $0\leq \gamma \leq N$ let denote by $\mathcal{L}^{r,\gamma}$ the usual homogeneous Morrey space. For $1\leq r < 2^\ast$ we let $\gamma =r{(N-2s)}/{2}$  (so that $0<\gamma< N$) and consider $X= \mathcal{L}^{r,{r(N-2s)}/{2}}$.  Note that with this choice of the parameters the norm of $X$ has the same invariance property of the $\dot{H}^s$ norm under~\eqref{def_funzioni} (see equation \eqref{def_morrey}  below). Our first result is the following

\begin{thm}
\label{thm_morrey1}
For any $0<s<N/2$ there exists a constant $C$ depending only on $N$ and $s$ such that, for   any $\dys \, {2}/{2^\ast}\! \leq \theta < 1$ and for any $1\leq r < 2^\ast$,
\begin{equation}
\label{eq_improved}
\dys
\|u \|_{L^{2^\ast}} \, \leq \, C \|  u\|_{\dot{H}^s}^{\theta} \|u\|_{\mathcal{L}^{r,r{(N-2s)}/{2}}}^{1-\theta}
 \quad \forall u\in \dot{H}^s(\R^N) \, .
 \end{equation}
\end{thm}
Note that a simple application of H\"older inequality gives $L^{2^\ast} \hookrightarrow \mathcal{L}^{r, \frac{N-2s}{2}r}=X$,  i.~\!e., there exists a constant $C=C(n,s)$ 
such that
\begin{equation}\label{eq_6star}
\dys
\|u\|_{\mathcal{L}^{r, r\frac{N-2s}{2}}} \, \leq \, C \| u\|_{L^{2^\ast}}, \quad \forall u \in L^{2^\ast} \, ,
\end{equation}
thus \eqref{eq_improved} is a refinement of \eqref{eq_sobolev} in the sense discussed above.
  Note moreover that H\"older inequality also gives $\dys \mathcal{L}^{p, p\frac{N-2s}{2}} \hookrightarrow \mathcal{L}^{r, r\frac{N-2s}{2}}, $
for any $1 \leq r < p < 2^\ast$, so that it will be enough to prove the theorem in the case $r=1$.

We propose two different proofs of Theorem~\ref{thm_morrey1}.
 The first one  relies basically on 
a subtle estimate of the Riesz potentials on weighted $L^p$ spaces established in~\cite{SW92}, using Calder\'on-Zygmund type techniques much in the spirit of the fundamental Fefferman-Phong inequality. Combining this estimate with a precise control on $A_{p,q}$-constant associated to the weights in terms of the Morrey norm the theorem follows. 
 
The second proof we give is completely different and combines the refined Sobolev embedding of \cite{GMO97} in $\dot{B}^{-N/2^\ast}_{\infty,\infty}$ with an embedding of Morrey spaces into homogeneous Besov spaces. Indeed, by means of the thermic description of the Besov spaces $\dot{B}^{-\alpha}_{\infty, \infty}$ (see \cite[Chapter 5]{Lem02}, and Section~\ref{sec_besov} below), we will check that 
\begin{equation}\label{bmstar}
\dys 
\mathcal{L}^{1,\alpha} \hookrightarrow \dot{B}_{\infty,\infty}^{-\alpha} ,
\end{equation}
for any $0<\alpha<N$, and for $\alpha ={(N-2s)}/{2}$ this is all we need to conclude.

The latter embedding is somehow implicit in \cite{Tay92} and presumably well-known in the Navier-Stokes community (at least for nonhomogeneous spaces or for $\alpha=1$, for which the space has the same scale-invariance of the equations). Since we were not able to find a precise reference to the literature, we provide below an elementary proof (see~Lemma~\ref{lem_10}).

\vspace{1mm}

To summarize, for $0<2s<N$ and $1\leq r<2^\ast$ we have the chain of inclusions  
$$
\dot{B}^s_{2,\infty} \, \hookrightarrow  \, L(2^\ast, \infty) \,   \hookrightarrow   \, \mathcal{L}^{r, r\frac{N-2s}{2}} \,  \hookrightarrow  \, \dot{B}^{-N/2^\ast}_{\infty,\infty}  \, ,
$$
and the refined Sobolev inequality \eqref{refsobolev} holds when $X$ is any of these spaces, due to Theorem \ref{thm_morrey1} and all the previous results recalled above. 

Actually the same strategy above applies to the Sobolev embedding $\dot{W}^{1,p}(\mathds{R}^N) \hookrightarrow L^{p^\ast}(\mathds{R}^N)$, $1\leq p<N$, where $\dot{W}^{1,p}$ is the closure of $C^\infty_0(\mathds{R}^N)$ with respect to the $L^p$-norm of the gradient. To be precise, we will prove the following
\begin{thm}\label{thm_morrey2}
For any $1\leq p<N$ let  $p^\ast$ be the usual critical exponent given by $Np/(N-p)$. There exists a constant $C$ depending only on $N$ and $p$ such that, for any $\, p/p^\ast \leq \theta < 1$ and for any $1\leq r< p^\ast$,
\begin{equation*}
\dys
\|u \|_{L^{p^\ast}} \leq \, C \| \nabla u\|_{L^p}^{\theta} \|u\|_{\mathcal{L}^{r,r{(N-p)}/{p}}}^{1-\theta} 
 \quad \forall u\in \dot{W}^{1,p}(\R^N) \, .
 \end{equation*}
 \end{thm}
As in the nonlocal case given in Theorem~\ref{thm_morrey1} --  the previous result contains an improvement of the usual Sobolev embedding $\dot{W}^{1,p} \hookrightarrow L^{p^\ast}$ in terms of Morrey spaces. The previous refinement is similar to the one in \cite{Led03}, where a stronger improvement is established analogous to \eqref{refsobolev} using the Besov space of negative smoothness $X=\dot{B}^{(p-N)/p}_{\infty,\infty}$. As for Theorem \ref{thm_morrey1} we give two proofs of this result, again either by weighted estimates of the Riesz potentials or combining the refined Sobolev inequality from \cite{Led03} with the embedding \eqref{bmstar} with $\alpha=(N-p)/p$.
\vspace{2mm}

Armed with the improved Sobolev embeddings in Theorem~\ref{thm_morrey1}, we will prove that
for any sequence $\{u_n\}$ in $\dot{H}^s$ uniformly bounded from below in the Lebesgue $L^{2^\ast}$-norm one can detect an appropriate scaling $\{x_n, \lambda_n\}$ which assures that the sequence~$u_{x_n,\lambda_n}$ given by~\eqref{def_funzioni} admits a nontrivial weak-limit (see Lemma~\ref{lem_nontrivial}). Combining this fact with a celebrated lemma from \cite{brezislieb}, we can study maximizing sequences for~\eqref{pbsobolev}.  As for the case when $s$ is an integer, first considered in \cite{lions}-\cite{lions2}, we will prove that the compactness of such sequences is restored when the natural invariance is taken into account. Indeed, we have the following

\begin{thm}
\label{the_optimizseq}
Let  $\{u_n\} \subset \dot{H}^s(\mathds{R}^N)$ be a maximizing sequence for the critical Sobolev inequality in the form~{\rm\eqref{pbsobolev}}. Then, up to subsequences, there exist a sequence of points $\{x_n \}\subset \mathds{R}^N$ and a sequence of numbers $\{\lambda_n\} \subset (0,\infty)$ such that $\tilde{u}_n(x)=\lambda_n^{(N-2s)/2}u_n \left(x_n+\lambda_n x\right)$ converges to $u(x)$ as given by \eqref{def_talentiana}, both in $L^{2^*}\!(\mathds{R}^N)$ and in $\dot{H}^s(\mathds{R}^N)$ as $n \to \infty$.
\end{thm}
As a consequence of the previous theorem, 
we see that optimizing sequences for the Sobolev inequality \eqref{pbsobolev}, at least asymptotically,  look like optimal functions.
It would be interesting to prove a quantitative version of this fact in analogy with what is done in \cite{BE} for the case $s=1$.

\vspace{2.2mm}

As next step, we informally\footnote{ We defer to \cite[Chapter 3]{TF07}, and to Section \ref{sec_decomposition} below for precise definitions.} review the notion of profile decomposition in general Hilbert spaces and we present an alternative abstract approach, based on \cite{TF07}  and Theorem \ref{thm_morrey1}, to the profile decomposition in $\dot{H}^s(\mathds{R}^N)$ as first proved in \cite{Ger98} by different arguments. Consider a bounded sequence $\{ u_n \} \subset H$ weakly converging to some $u \in H$, where $H$ is a given separable Hilbert space (in our case, $H=\dot{H}^s$, $0<s<N/2$) and let $u=0$ without loss of generality. We assume that some noncompact group $G$ acts by unitary operators on $H$ (in our case $G=\mathds{R}^N  \rtimes (0,\infty)$ acts by traslations and dilations according to \eqref{def_funzioni}), where the elements of the group~$G$ (possibly just a set) are sometimes called ``dislocations''.

 One can define the set of all {\em profiles} associated to $\{u_n\}\subset H$ as the set of all possible nonzero weak limits 
\begin{equation}
\label{defprofiles}
\bold{\Psi}=\{\psi_j \in H\setminus \{0\} , \, j \in I;  \, \psi_j= w- \lim_{n\to \infty} \, \left( g^{(j)}_n \right)^* u_n \, , \,  g_n^{(j)} \in  G  \}\,,
\end{equation}      
where $I$ is the (at most countable) index set for the profiles, and for each $j \in I$ the sequence $\{ g^{(j)}_n\} \subset G$ is going to infinity on the group. 
In case such set is not empty, one can try to subtract-off from $u_n$ all the profiles (scaled-back in the opposite way they where constructed) and analyze the asymptotic properties of the reminders.
More precisely, one writes

$$
u_n=\sum_{j \in I} g^{(j)}_n \psi_j +r_n \, ,
$$
where the vectors $\{  \psi_j\}_{j \in I} \subset H$, the dislocations $\{ g^{(j)}_n\}\subset G$  and the reminders $\{ r_n\} \subset H$ are characterized by three requirements: (i) the profiles $\psi_j$ are nonzero (nontriviality),
 (ii) for different $j \in I$ the corresponding dislocations are different as $n \to \infty$ (asymptotic orthogonality), 
 (iii) the sequence of reminders $r_n$ contains non further profiles as $n \to \infty$ ($G$-weak convergence to zero).

According to \cite[Theorem 3.1]{TF07}, such abstract decomposition is always possible and essentially unique (see also \cite{Tao10} for another proof). 
Thus, one can think of the investigations of profile decompositions as an attempt to capture the main features of the Banach-Alaoglu theorem in the extended setting of a group $G$ acting on $H$. Moreover, the presence of profiles is an obvious obstruction to compactness for the embedding of $H$ in any Banach space $Y$ on which the group $G$ acts also by isometries in an equivariant way (e.~\!g.  when $H=\dot{H}^s$ we can take $Y=L^{2^{\ast}}$ in view of~\eqref{eq_sobolev} and~\eqref{def_funzioni}). For discussing the relevance of this theory with references, proofs and some explicit applications, we refer the interested reader to the notes by Tao in~\cite{Tao08}.

Since such an abstract decomposition is available, the difficulty in applying it to $\dot{H}^s$ is just to characterize property (ii) and (iii) in a concrete way, and proving (iii) is precisely the point where improved Sobolev embeddings \eqref{eq_improved} enter. More precisely, using Theorem \ref{thm_morrey1} it is very easy to  show that $G$-weak convergence to zero is equivalent to strong convergence to zero in $L^{2^\ast}$.

Thus, combining Theorem 3.1 and Corollary 3.2 from \cite{TF07} with Theorem \ref{thm_morrey1}, we recover the following result by Gerard (see \cite{Ger98}, Theorem 1.1).

\begin{thm}\label{thm_decomposition}
Let $\{ u_n\}$ be a bounded sequence in $\dot{H}^s(\R^N)$. Then, there exist a {\rm(}at most countable{\,\rm)} set $I$, a family of profiles $\{\psi_j\}\subset \dot{H}^s(\R^N)$, a family of points $\{x^{(j)}_n \}\in \R^N$ and a family of numbers $\{ \lambda_n^{(j)}\}\subset(0,\infty)$, such that, for a renumbered subsequence of $\{ u_n\}$, we have
\begin{equation*}
\dys
\left| \log{\left(\frac{\lambda^{(i)}_n}{\lambda^{(j)}_n}\right)} \right| + \left|\frac{(x^{(i)}_n - x^{(j)}_n)}{\lambda^{(i)}_n}\right| \,
 \underset{n\to\infty}\longrightarrow \infty \ \ \text{for} \ i \neq j,
\end{equation*}
\begin{equation}\label{decomposition}
\dys
u_n(x) = \sum_{j\in I} {\lambda^{(j)}_n}^{\frac{2s-N}{2}} \psi_j\left(\frac{x-x^{(j)}_n}{\lambda_n^{(j)}}\right) + r_n(x),
\end{equation}
where
$\dys \,  \lim_{n\to \infty} \|r_n\|_{L^{2^\ast}} = 0$,\vspace{1mm} 
\begin{equation*}
\dys
\hspace{-2.8cm} \text{and} 
\qquad\qquad\quad \ \ \|u_n\|^2_{\dot{H}^s} = \sum_{j\in I} \| \psi_j\|^2_{\dot{H}^s} + \|r_n\|^2_{\dot{H}^s} + o(1) 
\ \text{as} \ n \to \infty.
\end{equation*}
\end{thm}

The first form of profile decomposition appeared in~\cite{Str84}, when $s=1$ in analyzing the failure of Palais-Smaile condition, under the name of ``global compactness'' property and involving finitely many profiles (see also \cite{HR} for the same result when $s=2$). Almost at the same time, still for $s\geq 1$ integer, a kind of profile decompositions in the sense of measures for bounded sequences in Sobolev spaces has been given in \cite{lions,lions2} and this aspect will be discussed below.
The first general result, close to Theorem \ref{thm_decomposition} is in \cite{Sol95}, still for $s=1$, using an improved Sobolev embedding in Lorentz spaces. 

Few years later, in the remarkable paper \cite{Ger98}, the author proved Theorem \ref{thm_decomposition} combining, among other things, a subtle analysis of $h-$oscillating sequences in $L^2(\mathds{R}^N)$ and a tricky exhaustion method with the refined Sobolev inequality \eqref{refsobolev} for $X=\dot{B}^s_{2,\infty}$. The specific choice of $X$ is crucial in \cite{Ger98}, among other things, to characterize the absence of profiles in a given sequence. 
Then, profile decompositions in $\dot{H}^s$ spaces has become a common decisive tool in the study of properties of solutions of many evolution equations and related issues  (see e.~\!g.  \cite{BaGe}, \cite{KPV}, \cite{MV}, \cite{Gal01}, \cite{KM06}, \cite{FVV} and the references therein).

Some time after \cite{Ger98}, an abstract approach in general Hilbert space appeared (see \cite{TF07} and the references therein), yielding profile decomposition in $\dot{H}^s$, $s$ integer, in a much simpler way. 
\vspace{1.5mm}

Our contribution here is twofold: on the one hand, following \cite{TF07}, we recover the decomposition result in an easier and more transparent way, on the other hand, in contrast with \cite{Ger98}, we show that absence of profiles can be actually characterized in terms of the much simpler spaces $\mathcal{L}^{r, r\frac{N-2s}{2}}$, for any $1\leq r<2^\ast$ (see Corollary \ref{nofprofiles}). 

It should be also mentioned that, after the paper~\cite{Ger98}, there have been some extensions of profile decompositions to more general Banach spaces, namely to Bessel spaces $\dot{H}^{s,p}$ (\cite{Jaf99}) and, more recently, to Besov spaces $\dot{B}^s_{p,q}$ (\cite{Koc10}). In both cases, the decomposition is heavily based on the construction of concrete unconditional bases  in terms of wavelets. It would be very challenging to try to develop a more general approach in Banach spaces to recover these results in a simpler abstract way.

\vspace{2mm}

Now, in order to study the behavior of a maximizing sequence for \eqref{pbsobolev} and \eqref{pbsobolev2} it is also convenient to establish a concentration-compactness alternative for bounded sequences in the fractional space $\dot{H}^s$ in terms of measures, using methods and ideas introduced in the pioneering works \cite{lions} and \cite{lions2} and developed extensively in literature (see, e.~\!g.,
\cite{flucher}, 
\cite{TF07} and the references therein). 
\vspace{1mm}

We have the following
\begin{thm}\label{the_cca}
Let $\Omega \subseteq \mathds{R}^N$ an open subset and let $\{u_n\}$ be a sequence in $\dot{H}^s(\Om)$ weakly converging to $u$ as $n \to \infty$ and such that
$$
|(-\Delta)^{\frac{s}{2}} u_n|^2dx \tows \mu \ \ \ \text{and} \ \ \ |u_n|^{2^{\ast}}dx\tows \nu \ \ \text{in} \ \mea.
$$ 
Then, either $u_n \to u$ in $L^{2^{\ast}}_{\rm{loc}}(\mathds{R}^N)$ or  there exists a (at most countable) set of distinct points $\{x_j\}_{j\in J} \subset \Omb$ and positive numbers $\{\nu_j\}_{j\in J}$ such that we have
\begin{equation}
\label{quantnu}
\nu=\ |u|^{2^{\ast}}dx+\sum_{j} \nu_j \delta_{x_j}.
\end{equation}
If, in addition, $\Omega$ is bounded, then there exist a positive measure $\tilde{\mu} \in \mathcal{M}(\mathds{R}^N)$ with {\rm spt}~$\!\tilde{\mu}~\subset~\Omb$ and positive numbers $\{\mu_j\}_{j\in J}$  such that
\begin{equation}
\label{quantmu}
\mu=|(-\Delta)^{\frac{s}{2}}u|^2dx+\tilde{\mu}+\sum_{j} \mu_j \delta_{x_j}, \quad \nu_j \leq S^{\ast} (\mu_j)^{\!\frac{2^{\ast}}2}\, .
\end{equation}

\end{thm}
\vspace{1mm}

The previous result extends to the case of the fractional spaces $\dot{H}^s$ a well known fact for $s=1$ and, more generally, when $s$ is an integer (see \cite{lions} and~\cite{lions2}; see also~\cite{TF07} and the references therein); namely that, at least locally, compactness in the Sobolev embedding fails precisely because of concentration of the $L^{2^{\ast}}\!$ norm at countably many points. 
These results have been largely used for the variational treatment of the Yamabe problem and their higher order analogues involving the Paneitz-Branson operators and more generally for semi-linear elliptic equations with critical nonlinearities.
As notice in~\cite{Ger98}, when  $\Omega=\R^N$, a different proof of~Theorem~\ref{the_cca} can be also deduced as a byproduct of the 
profile-decomposition in Theorem~\ref{thm_decomposition} (for a possibly different index set $J$).
In this respect, in~\eqref{quantmu} the sum of Dirac masses comes from those profiles in~\eqref{decomposition} which are peaking at the $x_j$'s.

We will provide a simple proof of Theorem~\ref{the_cca} by following the original argument in~\cite{lions} and~\cite{lions2}; clearly, we need to operate some modifications due to the  non-locality of  $(-\Delta)^{\frac{s}{2}}$. Indeed, our approach relies on pseudodifferential calculus to control the natural error term in the localization by cut-off functions. Using a simple commutator estimate (see, e.~\!g., Taylor \cite{t}) and a standard approximation argument, we will show the compactness of the commutator $[\varphi, (-\Delta)^{\frac{s}{2}}]:\dot{H}^s(\Omega) \to L^2(\mathds{R}^N)$ when $\varphi\in C^{\infty}_0(\R^N)$, at least if $\Omega$ is bounded (see Lemma \ref{lem_commutator}).  As a consequence we will give local description of the lack of compactness in $L^{2^\ast}$ in terms of atomic measures. We hope that these results will be also of use in the variational theory of the fractional Yamabe problem firstly considered in \cite{GQ}.
\vspace{2mm}

As a corollary of Theorems  \ref{the_optimizseq} and \ref{the_cca}, we will  see that concentration always occurs in problem \eqref{pbsobolev2} because of the classification in~\eqref{def_talentiana}; 
see Corollary~\ref{cor_concentrazio-sob}. Existence/nonexistence of optimal functions in problems \eqref{pbsobolev} and \eqref{pbsobolev2} could be studied for other equivalent norms. Even for norms equivalent to \eqref{Hs0-norm} (and analogously to \eqref{gagliardo} and \eqref{tracesobolev}-\eqref{fracextension} defined in Section~\ref{sec_fractional}), e.~\!g.~obtained multiplying by suitable kernels $|a(\xi)|\,, \, \, |K(x,y)|$ and $|A(x,t)|$ bounded from above and below, we expect the existence of optimal function to depend in a nontrivial subtle way on the chosen kernels (see \cite{MP} for similar results in this direction).

\vspace{3mm}

Finally, we consider a family of problems for perturbations of the functional \eqref{funzionale2}. Let $0<\varepsilon<2^{\ast}-2$ and let $\Omega \subset \mathds{R}^N$ be a bounded open set. We set 
\begin{eqnarray}\label{problemaeps}
\dys S^{\ast}_{\eps} \! := \!\sup\left\{ F_{\eps}(u) : u \in \dot{H}^s(\Omega), \, \int_{\R^N} |(-\Delta)^{\frac{s}{2}} u|^{2} dx \leq 1 \right\}
\end{eqnarray}
\begin{eqnarray*} 
\dys \text{where} \ \ F_\eps (u)\!\! := \!\!\int_{\Om} |u|^{2^{\ast}\!-\eps} dx.  
\end{eqnarray*}
The previous maximization problems are subcritical. Indeed, since $\Omega$ is a bounded open set and the  embedding $\dot{H}^s(\Omega) \hookrightarrow L^{2^{\ast}-\varepsilon}(\Omega)$ is compact, the previous problem admits a maximizer $u_\varepsilon\in \dot{H}^s(\Omega)$. Our purpose is to investigate what happens when $\varepsilon~\to~0$ both to the subcritical Sobolev constant $S^{\ast}_\eps$ (i.~\!e., the optimal constant for the embedding  $\dot{H}^s(\Omega) \hookrightarrow L^{2^{\ast}-\varepsilon}(\Omega)$ given in~\eqref{problemaeps})
and to the corresponding maximizers $u_\varepsilon$ (i.~\!e. the corresponding optimal functions).
\vspace{2mm}

Combining Theorem~\ref{the_optimizseq} together with Theorem~\ref{the_cca}, we have

\begin{thm}\label{the_concentrazione1}
Let $\Omega \subset\mathds{R}^N$ be a bounded open set and for each $0<\eps<2^\ast-2$ let $\ue\in \dot{H}^s(\Omega)$ be a maximizer for $S^{\ast}_{\eps}$.  Then
\begin{itemize}
\item[(i)]{$\dys \lim_{\eps\to0} S^{\ast}_{\eps}=S^*$;
}\vspace{2mm}
\item[(ii)]{As $\varepsilon=\varepsilon_n \to 0$, up to subsequences $u_n={\ue}_n$ satisfies $u_n \rightharpoonup 0$ in $\dot{H}^s(\Omega)$ and it concentrates at some point $x_{0}\in\Omb$ both in $L^{2^{\ast}}$\! and in $\dot{H}^s$, i.~\!e.
$$
|u_n|^{2^{\ast}}dx \tows S^{\ast} \delta_{x_0} \ \, \text{and} \  \, |(-\Delta)^{\frac{s}{2}}u_n|^{2}dx \tows \delta_{x_{0}} \ \text{in} \ \mathcal{M}(\mathds{R}^N).
$$
}
\item[(iii)]{
There exists a sequence of points $\{x_n \}\subset \mathds{R}^N$, $x_n\to x_0$ and a sequence of numbers $\{\lambda_n\} \searrow 0$, such that $\tilde{u}_n(x)=\lambda_n^{(N-2s)/2}u_n\left(x_n+\lambda_n x \right)$ converges to $u(x)$ as given by \eqref{def_talentiana}, both in $L^{2^*}\!(\mathds{R}^N)$ and in $\dot{H}^s(\mathds{R}^N)$ as $n \to \infty$.
}
\end{itemize}
\end{thm}

The previous concentration result is well known for $s=1$.
The asymptotic behavior of the optimal functions has been discusses in~\cite{Han91} and~\cite{rey}, at least assuming (i) and the smoothness of the domain~$\Om$.
For the case of general possibly non-smooth domains we refer to~\cite{Pal11b}. It would interesting  to characterize the concentration point~$x_0$ as critical point of some function.
This is known to be the case when $s=1$ or $s=2$, the function being the regular part of the Green function of the Laplacian or the BiLaplacian in the domain~$\Om$  (see \cite{Han91}, \cite{rey} and \cite{ChGe}).

\vspace{2mm}

Here, we also note that the maximizers~$\ue \in \dot{H}^s(\Om)$ discussed in~Theorem  \ref{the_concentrazione1} are in fact solutions of the semi-linear equation
\begin{equation}\label{lagrange}
(-\Delta)^{s} u_{\eps} = \lambda |u_{\eps}|^{2^{\ast}-2-\eps}\ue \ \ \text{in} \ (\dot{H}^s(\Omega))',
\end{equation}
where $\lambda=(S^{\ast}_\eps)^{-1}$ is a Lagrange multiplier. Indeed, \eqref{lagrange} is the Euler-Lagrange equation for the functional $F_\varepsilon$ among functions with $\dot{H}^s$ norm equal to one. 
Our results yield a concentration phenomenon for a sequence of solutions $\ue$, as $\varepsilon\to 0$.  

In this respect, another critical problem that would be very natural to investigate is
\begin{equation}\label{brezisnirenberg}
(-\Delta)^{s} u -\eta u=  |u|^{2^{\ast}-2}u \ \ \text{in} \ (\dot{H}^s(\Omega))',
\end{equation}
where $\eta>0$ is a parameter. Well known results for $s=1$ (see \cite{BrNi}) and $s=2m$ an even integer (see~\cite{EFJ} and~\cite{PS}) suggest that, even for fractional values of $s$, existence results for \eqref{brezisnirenberg} should always depend in a delicate way on $\eta$ (see,~e.~\!g., the forthcoming paper~\cite{SV13} for first results when $s\in (0,1)$).

\vspace{2mm}

The rest of the paper is organized as follows. In Section~\ref{sec_spaces} we briefly recall the definitions and some basic properties of the function spaces we deal with, also analyzing the corresponding scaling properties. In Section~\ref{sec_improved} we will prove the refined Sobolev embeddings in terms of Morrey spaces, Theorem~\ref{thm_morrey1} and Theorem~\ref{thm_morrey2}. In Section~\ref{sec_maxi}, we study the optimizing sequences for the Sobolev embedding \eqref{eq_sobolev}, proving Theorem~\ref{the_optimizseq}. Section~\ref{sec_decomposition} is devoted to the application of the abstract $G$-weak  convergence to profile decomposition in the spaces $\dot{H}^s$, as given by Theorem~\ref{thm_decomposition}. In Section \ref{sec_cca}, we prove Theorem~\ref{the_cca} by establishing the concentration-compactness alternative in terms of measures and we discuss maximizing sequences for the Sobolev inequality~\eqref{pbsobolev2} (see Corollary \ref{cor_concentrazio-sob}). Finally, we analyze the asymptotic behavior of the subcritical Sobolev constant $S^\ast_\eps$ and the corresponding optimal functions proving Theorem~\ref{the_concentrazione1}.

\vspace{1mm}

\section{Relevant function spaces}\label{sec_spaces}

Throughout the paper, $N$ will always be the dimension of the ambient space and will be greater or equal than 1. {F}or any real $0<s<N/2$, we denote by
\begin{equation}\label{def_2star}
2^{*}\equiv 2^*_s:= 2N/(N-2s)
\end{equation}
the standard critical Sobolev exponent.
\vspace{1mm}

Also, we follow the usual convention of denoting by $c$ a general positive constant 
that may vary from line to line. Relevant dependencies will be emphasized by using parentheses or special symbols.

\vspace{1mm}

As usual, we denote by 
$$
B_R(x_0)=B(x_0;R):=\{x\in \mathds{R}^N : |x-x_0|<R\}
$$
the open ball centered in $x_0\in \mathds{R}^N$ with radius $R>0$. When not important and clear from the context, we shall use the shorter notation $B_R=B_R(x_0)$.

\subsection{Fractional Sobolev spaces}\label{sec_fractional}

{F}or each $s\geq 0$ let
$$
H^s(\mathds{R}^N)=\big\{ u\in L^2(\mathds{R}^N) \, \,\hbox{s.\!~t.} \, \, |\xi|^{s}\hat{u}(\xi) \in L^2(\mathds{R}^N)\, \big\}
$$
be the standard {\em fractional Sobolev space} $H^s$ defined using the Fourier transform
$$
\dys \mathcal{F}(u)(\xi)~=~\hat{u}(\xi) \, =\, \frac{1}{(2\pi)^{\frac{N}2}}\int_{\R^N}e^{-ix\cdot\xi}u(x)\,dx.
$$
\vspace{1mm}

As usual, the space $H^s(\mathds{R}^N)$ can be equivalently defined as the completion of $C_0^\infty(\mathds{R}^N)$ with respect to the norm 
\begin{equation}
\label{Hs-norm}
\| u\|^2_{H^s}=\| (Id-\Delta)^{\frac{s}{2}} u\|^2_{L^2}=\int_{\mathds{R}^N}(1+|\xi|^2)^s|\hat{u}(\xi)|^2 d\xi \, , 
\end{equation}
where the operator $(Id-\Delta)^{\frac{s}{2}}=\mathcal{F}^{-1}\circ M_{(1+|\xi|^2)^{s/2}} \circ \mathcal{F}$ is conjugate to the multiplication operator on $L^2(\mathds{R}^N)$ given by the function $(1+|\xi|^2)^{s/2}$.

It is well known that for $0<s<N/2$, 
 the Sobolev inequality \eqref{eq_sobolev} is valid
for an explicit positive constant~$S^{\ast}=S^{\ast}(N, s)$ given by \eqref{def_sstar},
and for any function $u \in H^s(\mathds{R}^N)$. In order to discuss inequality \eqref{eq_sobolev}, it is very natural to consider for each $0\!~\!<\!~\!s\!~\!<\!~\!N/2$ the homogeneous Sobolev space
\begin{equation}\label{def_hs0}
\dys \dot{H}^s(\mathds{R}^N)=\big\{ u\in L^{2^{\ast}}(\mathds{R}^N) \, \,\hbox{s.\!~t.} \, \, |\xi|^{s}\hat{u}(\xi) \in L^2(\mathds{R}^N)\, \big\}.
\end{equation}
This space can be equivalently defined as the completion of $C_0^\infty(\mathds{R}^N)$ under the norm~\eqref{Hs0-norm}
and inequality \eqref{eq_sobolev} holds by density on $\dot{H}^s(\mathds{R}^N)$.
\vspace{2mm}

When $0<s<1$, a direct calculation using Fourier transform (see, e.~\!g., \cite[Proposition~3.4]{DPV12}) gives
\begin{equation}
\label{gagliardo}
\int_{\mathds{R}^N}|\xi|^{2s}|\hat{u}(\xi)|^2 d\xi=c(N,s) \int_{\mathds{R}^N} \int_{\mathds{R}^N} \frac{|u(x)-u(y)|^2}{|x-y|^{N+2s}}\, dx\,dy \, , 
\end{equation}
which provides an alternative formula for the norm on $\dot{H}^s(\mathds{R}^N)$. The previous equality fails for $s\geq 1$, since in that case the right hand-side in \eqref{gagliardo} is known to be finite if and only if $u$ is constant (see \cite{Bre02}).

When $0<s<1$, according to \cite{CS07} (see also \cite{CG11} for the more difficult case $1<s<N/2$, $s\not \in \mathds{N}$), the Sobolev inequality \eqref{eq_sobolev} is also equivalent to the trace Sobolev embedding $H^1_0(\mathds{R}^N\times [0,\infty),\, t^{1-2s}\,dx\,dt) \hookrightarrow L^{2^*}(\mathds{R}^N)$. Indeed, taking for simplicity $u \in C^\infty_0(\mathds{R}^n)$ and  $U\in C^\infty_0(\mathds{R}^N\times [0,\infty))$ such that $U(x,0)\equiv u(x)$ we have 
\begin{equation}
\label{tracesobolev}
\| u\|^2_{L^{2^*}(\mathds{R}^N)}\leq \, {S^{\ast}}^{2/2^*}\!\! \int_{\mathds{R}^N}|\xi|^{2s}|\hat{u}(\xi)|^2 d\xi 
\, \leq\, C(N,s) \int_{\mathds{R}^N}\int_0^\infty |\nabla U|^2 t^{1-2s} \,dx\,dt \, ,
\end{equation}
which extends to a bounded trace operator $T_r:H^1_0\to L^{2^\ast}$.
Moreover, the second inequality in \eqref{tracesobolev} is an equality if and only if the extension $U$ satisfies
\begin{equation}
\label{fracextension}
\left\{
\begin{array}{ll}
{\rm div} \left( t^{1-2s} \nabla U \right)=0 &   \hbox{in} \quad \mathds{R}^N \times (0,\infty) \, , \\
U(\cdot,0)=u & \hbox{in} \quad \mathds{R}^N \, .
\end{array}
\right.
\end{equation} 
Actually, the solution operator to \eqref{fracextension} allows to identify $\dot{H}^s(\mathds{R}^N)$ as the trace space of  $H^1_0(\mathds{R}^N\times [0,\infty), t^{1-2s}\,dx\,dt)$ and the Sobolev inequality \eqref{eq_sobolev} as the trace inequality in \eqref{tracesobolev}.

\vspace{1mm}

\subsection{Besov spaces}\label{sec_besov}
For any real $\alpha >0$ the homogeneous Besov space $\dot{B}^{-\alpha}_{\infty,\infty}$ is defined as the set of tempered distributions on $\mathds{R}^N$ (possibly modulo polinomials) such that 
$$
\sup_{j\in\Z} 2^{-\alpha j}\| (\Delta_j u)\|_{L^\infty} < \infty \, ,
$$
where $\Delta_j$ are the frequency localization operators at frequency of order $|\xi| \sim 2^{-j}$ associated to a standard Paley-Littlewood decomposition in frequency space.

In the rest of the paper, we make essential use of the equivalent thermic description for the homogeneous  Besov spaces with negative exponent above. 
To this aim, for each $t\geq 0$, we denote by
\begin{equation}\label{def_heat}
P_t:=e^{t\Delta}
\end{equation}
the standard heat semigroup on $\R^N$. For any real $\alpha >0$, one can equivalently define the homogeneous Besov space $\dot{B}^{-\alpha}_{\infty,\infty}$ as the space of tempered distributions $u$ on $\R^N$  (possibly modulo polinomials) for which the following norm
\begin{equation}\label{def_besovnorm}
\dys
\|u\|_{\bb} := \sup_{t>0} t^{\alpha/2} \|P_t u\|_{L^\infty}
\end{equation}
is finite.  The thermic description above is classic and various references are available; see, e.~\!g., Theorem~5.4 in the book~\cite{Lem02} by Lemari\'e-Rieusset.

\vspace{1mm}
The special case which is relevant to us is the Besov space of index $\alpha = N/2^{*}$, since this is the case when the norm in~\eqref{def_besovnorm} is invariant by dilation and translation with respect to the same scaling factor of the critical Lebesgue space~$L^{2^\ast}$ and the fractional Sobolev spaces $\dot{H}^s$ (see~\eqref{def_funzioni}). This is easy to check by computation. Indeed, take $u\in \dot{B}^{-N/2^\ast}_{\infty,\infty}$ and, for any $x_0\in\R^N$ and any $\lambda>0$, consider the function $u_{x_0,\lambda}$ given by~\eqref{def_funzioni}.
By definition in~\eqref{def_heat} together with the standard change of variable formula, 
it follows that
$$
\| u_{x_0,\lambda}\|_{\dot{B}^{-{N}/{2^\ast}}_{\infty,\infty}} \equiv \|u\|_{\dot{B}^{-{N}/{2^\ast}}_{\infty,\infty}}.
$$

\vspace{1mm}

\subsection{Morrey spaces}\label{sec_morrey}
We recall the definition of the {\it Morrey spaces}~$\mathcal{L}^{r,\gamma}$, introduced by Morrey  as a refinement of the usual Lebesgue spaces. 
A measurable function $u:\R^N\to\R$ belongs to the Morrey space~$\mathcal{L}^{r, \gamma}(\R^N)$, with $r\in [1,\infty)$ and $\gamma\in [0,N]$, if and only if
\begin{equation}\label{def_morrey}
\|u\|^r_{\mathcal{L}^{r, \gamma}(\R^N)} := \sup_{R>0;\ x\in \R^N} R^{\,\gamma} \dashint_{B_R(x)} |u|^r\, {\rm d}y \, < \, \infty.
\end{equation}
An equivalent definition can be provided by using cubes $Q\subseteq\R^N$ instead of balls.
By definition, 
one can see that  
 if $\gamma=N$ then the Morrey spaces $\mathcal{L}^{r, N}$ coincide with the usual Lebesgue spaces ${L}^r$ for any $r\geq 1$; similarly  $\mathcal{L}^{r,0}$ coincide with $L^{\infty}$.

\vspace{2mm}

It is worth noticing that the exponents $r$ and $\gamma$ in~\eqref{def_morrey} endow the spaces $\mathcal{L}^{r,\gamma}$ with the same dilation and translation invariance of the Lebesgue space $L^{2^\ast}$ -- and therefore of~$\dot{H}^s$ and~$\dot{B}^{-N/2^\ast}_{\infty,\infty}$ -- if (and only if) they are suitably coupled, namely $\gamma/r=N/{2^\ast}$. Indeed, for any $x_0\in \R^N$ and any $\lambda>0$, let $u_{x_0,\lambda}$ be the function defined by~\eqref{def_funzioni}. By change of variable formula, one can see that the following equality holds
\begin{equation*}
\dys
\| u_{x_0,\lambda} \|_{\mathcal{L}^{r, \frac{N-2s}{2}r}} \, = \, \| u \|_{\mathcal{L}^{r, \frac{N-2s}{2}r}},
\end{equation*}
for any $1 \leq r \leq 2^{\ast}$.
Also, in the rest of the paper, we make use of the following property, which one can plainly deduce by the definition in~\eqref{def_morrey}. For any $\beta>0$,
$$
\| |u|^\beta\|_{\mathcal{L}^{r,\gamma}} = \|u\|^\beta_{\mathcal{L}^{\beta r,\gamma}}. 
$$
Finally note that H\"older inequality readily yields the embedding $\mathcal{L}^{r,\gamma} \hookrightarrow \mathcal{L}^{1,\gamma/r}$ for any $r \in (1,\infty)$ and any $\gamma \in (0,N)$. Equivalently, with the same restriction on the parameters, there exists a constant $C=C(N,r,\gamma)>0$ such that 
\begin{equation}\label{lem_8}
\dys \| u \|_{\mathcal{L}^{1,\gamma/r}} \, \leq \, C \|u \|_{\mathcal{L}^{r,\gamma}},
\end{equation}
for all $u\in \mathcal{L}^{r,\gamma}(\R^N)$.

\section{Improved Sobolev inequalities}\label{sec_improved}

\subsection{Refinements via Riesz potentials}
This section is devoted to the proofs of the improved Sobolev embedding given in Theorems~\ref{thm_morrey1} and~\ref{thm_morrey2} based on the weighted estimates on Riesz potentials established in \cite{SW92}. A slightly different proof, yet involving the same kind of Calder\'on-Zygmund type techniques could be obtained using a bilinear estimate of the Riesz potentials on Morrey spaces proved in \cite{SST11}.
\vspace{1mm}

First, we recall the characterization of the fractional Sobolev space $\dot{H}^s(\R^N)$ in terms of Riesz potentials. For any $u\in \dot{H}^s(\R^N)$ let $g\in L^2(\mathds{R}^N)$  be  a function such that $\hat{g}(\xi) = |\xi|^s \hat{u}(\xi)$. In view of Plancherel's Theorem we have
\begin{equation}\label{g}
\dys
\| u\|_{\dot{H}^s} 
\, = \, \|(-\Delta)^{\frac{s}{2}} u \|_{L^2} 
\, = \, \|g\|_{L^2},
\end{equation}
On the other hand, using Fourier transform, we can write
\begin{equation}\label{eq_17stella}
u  =  \ff^{-1}\left(\frac{1}{|\xi|^s} |\xi|^s \hat{u}(\xi)\right)
 =  \ff^{-1} \left ( \frac{1}{|\xi|^s} \hat{g}(\xi)\right)
 =  \ff^{-1} \left( \frac{1}{|\xi|^s}  \right) \ast g=  \ii_s g,
\end{equation}
where, up to an explicit constant $c=c(N,s)$ that will be neglected in the rest of this section, the symbol $\ii_s g$ in~\eqref{eq_17stella} denotes the standard Riesz potential of order $s$, namely 
$$
\dys
\ii_s g(x) := \int_{\R^N} \frac{g(y)}{|x-y|^{N-s}}\, dy.
$$
As it is well known, the Sobolev inequality~\eqref{eq_sobolev} is equivalent to the fact that $\ii_s: L^2 \to L^{2^\ast}$ is well defined and it is a bounded operator.

\vspace{2mm}

The following norm inequality for Riesz potentials on weighted Lebesgue space  is the crucial ingredient in proving Theorem \ref{thm_morrey1}.

\begin{thm}\label{thm_vw}{\rm(see} \cite[Theorem 1]{SW92}{\rm)}.
Suppose that $0<s<N$, $1<p\leq q <\infty$ and that $V$ and $W$ are nonnegative measurable functions on $\R^N$, $N\geq 1$. If, for some $\sigma>1$,
\begin{equation}\label{eq_18stella2}
\dys |Q|^{\frac{s}{N}+\frac{1}{q}-\frac{1}{p}} \left( \dashint_Q V^\sigma dx\right)^{\frac{1}{q\sigma}}
\left( \dashint_Q W^{(1-p')\sigma} dx\right)^{\frac{1}{p'\sigma}} \, \leq \, C_\sigma,
\end{equation}
for all cubes $Q\subset\R^N$,
then for any function $f \in L^p(\R^N;W(x)dx)$ we have 
\begin{equation}\label{eq_18stella}
\dys 
\left( \int_{\R^N} \left| \ii_s f (x) \right|^q V(x) dx \right)^{\frac{1}{q}}
\, \leq \,
c\, C_\sigma \left( \int_{\R^N} \left| f (x) \right|^p W(x) dx \right)^{\frac{1}{p}},
\end{equation}
where the constant $c$ depends only on $p, q$ and $N$.
\end{thm}
Note that the statement of Theorem~1 in \cite{SW92} is slightly adapted to our purpose. The original statement involves only nonnegative functions, a case from which the statement above clearly follows considering positive and negative parts. 
In addition, in the original formulation the constants in~\eqref{eq_18stella} are actually presented as a unique constant~$c=c(p,q,N, V,W)$. The fact that it is a product of the separated constants in~\eqref{eq_18stella} can be easily deduced by a close inspection of the proof in~\cite{SW92}. 

\vspace{1mm}

\begin{proof}[\bf Proof of Theorem~\ref{thm_morrey1}]
The main idea in the present proof consists into a careful application of Theorem~\ref{thm_vw} to the function $g$ in~\eqref{g}, after interpolating the exponent~$2^\ast$ via a suitable tuning.
\vspace{1mm}
Choosing $W\equiv 1$ and $p=2$, the assumption in~\eqref{eq_18stella2} can be rewritten as 
\begin{equation}\label{eq_19stella}
\dys
|Q|^{\frac sN +\frac 1q -\frac12} \left(\dashint_Q V^\sigma dx\right)^{\frac1{\sigma q}}
\, \leq \, C_\sigma \, ,
\end{equation}
for some exponent $\sigma>1$ to be determined.

Now, for any fixed $2\leq  q<2^\ast$ and $u\in \dot{H}^s(\R^N)$ , we can chose the weight $V$ as $V(x) = |u(x)|^{2^\ast-q}$, so that 
\begin{equation}\label{eq_19stella2}
|u(x)|^{2^\ast}
\, = \, |u(x)|^q |u(x)|^{2^\ast-q} = |u(x)|^q V(x), \quad \forall x\in \R^N.
\end{equation}
\vspace{1mm}
Hence, taking $\sigma=\sigma(q):=1/(2^\ast-q)>1$, we have the obvious equality $\dys V(x)^\sigma=|u(x)|$, and the corresponding restriction on the tuning exponent
\begin{equation}\label{eq_c}
\max\big\{ 2, 2^\ast-1\big\} \, < \, q < \, 2^\ast,
\end{equation}
which will allow us to apply \eqref{eq_18stella} if \eqref{eq_18stella2} holds. Notice also that in high dimensions we can actually take any $q$ in $[2,2^\ast)$ because $2^\ast-1<2$.

Going back to \eqref{eq_19stella}, up to a positive constant  depending only on  $N,p,q,s$ (when passing from balls to cubes), we have
\begin{eqnarray}\label{eq_20stella}
\dys
|Q|^{\frac sN +\frac 1 q -\frac12}\left( \dashint_Q |u| \,dx\right)^{\!\frac{2^\ast-q}q}
\!& = &\! \left ( R^{\sigma(sq+N-\frac{Nq}{2})}\dashint_{B_R}|u|\, dx\right)^{\!\frac{1}{\sigma q}} 
\, = \, \left ( R^{\frac{N-2s}{2}}\dashint_{B_R}|u| \, dx\right)^{\!\frac{1}{\sigma q}}
\nonumber \\[1ex]
& \leq & \!\|u\|^{\frac{2^\ast-q}{q}}_{\mathcal{L}^{1,\frac{N-2s}{2}}}=:C_\sigma.
\end{eqnarray}
Keeping in mind the inequality above and the decomposition in~\eqref{eq_19stella2}, we can apply Theorem~\ref{thm_vw}, choosing in particular $p=2$ there. We get
$$
\dys
\int |u|^{2^\ast} dx
\ = \ \|\ii_s g\|^q_{L^q_V} 
\ \leq \ (c\, C_\sigma)^q \|g\|^q_{L^2}
\ = \ c\|u\|^{2^\ast-q}_{\mathcal{L}^{1,\frac{N-2s}{2}}} \|u\|^q_{\dot{H}^s}, 
$$
which plainly yields
\begin{equation}\label{39b}
\|u\|_{L^{2^\ast}} \, \leq \, C \|u\|_{\dot{H}^s}^\theta \|u\|^{1-\theta}_{\mathcal{L}^{1,\frac{N-2s}{2}}},
\end{equation}
for any exponent $\theta=q/2^\ast$ such that $\max\big\{{2}/{2^\ast}, 1\!-\!{1}/{2^\ast}\big\} 
 <  \theta  <  1$, as desired. Thus, if $r=1$ we have the conclusion, otherwise we also apply \eqref{lem_8} with $\gamma=r (N-2s)/2$ and the proof is complete.
\end{proof}

\vspace{1mm}

The preceding proof can be extended with simple modifications to cover the case of functions in $\dot{W}^{1,p}$.
 As a consequence we obtain the improved Sobolev inequality given by Theorem~\ref{thm_morrey2}.
We would like to notice that in such case we cannot expect to use an analog representation of Sobolev functions via Riesz potentials as the one given at the beginning of this section for the nonlocal case. Anyway, as it is well known, a pointwise
potential estimate always holds and this will suffice to extend to the nonlinear case the strategy via weighted Lebesgue inequalities presented before.

\vspace{1mm} For  $N\geq3$ (the case $N=2$ being analogous via the logarithmic kernel) and $u \in C^\infty_0(\R^N)$, again neglecting a multiplicative constant $c=c(N)$, we have
$$
\dys
u(x)  \, = \,  \Delta^{-1}\Delta u \, = \, \int\frac{1}{|x-y|^{N-2}} \Delta u(y) \, dy
\, = \, (2-N) \int \frac{x-y}{|x-y|^N}\nabla u(y)\, dy,
$$
which yields
\begin{equation}\label{eq_22stella}
\dys
|u(x)| \, \leq \, c \int\frac{|\nabla u(y)|}{|x-y|^{N-1}}\, dy
\, = \, c \,\ii_1(|\nabla u|)(x) \quad \text{for} \ x \in \R^N,
\end{equation}
and the same pointwise bound easily extends by density for any $u \in \dot{W}^{1,p}(\R^N)$.
\vspace{1mm}
\begin{proof}[\bf Proof of Theorem~\ref{thm_morrey2}]
We assume $1<p<N$. By taking into account the estimate in~\eqref{eq_22stella}, we can proceed as in the proof of Theorem~\ref{thm_morrey1}. Hence, for any $p \leq q<p^\ast$ it suffices to decompose the function $|u|^{p^\ast}$ as follows,
$$
|u(x)|^{p^\ast} = |u(x)|^q V(x),
$$
where $\dys V(x):= |u(x)|^{p^\ast-q}$ will be the weight to be chosen in~\eqref{eq_18stella2}, with the existence of the exponent $\sigma=1/(p^\ast-q) >1$ assured by taking $\max\big\{p, p^\ast -1\big\} < q < p^\ast$; again any choice $p \leq q <p^\ast$ is admissible in high dimensions, since $p^\ast-1<p$ in such case (recall the observation regarding the range in~\eqref{eq_c}).
\vspace{1mm}

We can deduce the dependence of the constant $C_\sigma$ as at the end of proof of Theorem~\ref{thm_morrey1}, and if $\theta=q/p^\ast$ from ~\eqref{eq_18stella} we finally get
\begin{equation}\label{eq_23stella}
\dys
\|u\|_{L^{p^\ast}} \, \leq \, C \|u \|^\theta_{\dot{W}^{1,p}} \| u \|^{1-\theta}_{\mathcal{L}^{1,\frac{N-p}{p}}},
\end{equation}
where $\max \big\{ p/p^\ast, 1 - 1/p^\ast \big\} < \theta < 1$. Possiby applying \eqref{lem_8} the proof is complete.
\end{proof}

\vspace{1mm}

Again, it is worth noticing that~\eqref{eq_23stella} constitutes a refinement of the usual Sobolev inequality $\|u\|_{L^{p^\ast}} \leq c \|u\|_{\dot{W}^{1,p}}$ because of the  embedding
$
\mathcal{L}^{r,r\frac{N-p}{p}} \hookrightarrow L^{p^\ast},
$
 for any $1 < p < N$ and any $1\leq r<p^\ast$, which follows from H\"older inequality.

\vspace{1mm}

\begin{rem}{\em
In both the proofs of Theorem~\ref{thm_morrey1} and Theorem~\ref{thm_morrey2}, a suitable choice of the involved weights and parameters will allow us to make use of Theorem~1.1 in \cite{SST11} in substitution to Theorem~\ref{thm_vw}. The deduced refined inequalities will be exactly as proved before.
}\end{rem}

\begin{rem} {\em Finally, it is worth noticing that the range of validity for the tuning parameter $\theta$, both in~\eqref{39b} and in~\eqref{eq_23stella}, is somehow restricted, unless $N$ is large, a limitation that is absent in the second proof that will be presented in the next subsection. 
We remark that, by simple modifications of the argument below, one could get the full range of convexity exponent $\theta$ but some restriction on the integrability exponent $r$, still unless $N$ is large. 
Similarly, a further restriction appeared in the proof of Theorem~\ref{thm_morrey2}, where we had to assume $p>1$ in order to apply Theorem~\ref{thm_vw}.
In all respects, the approach via Besov spaces at present seems to be more efficient.} \end{rem} 

\subsection{Refinements via Besov spaces}

As mentioned in the introduction, we can deduce the
improved Sobolev inequality in the Morrey scale given by~\eqref{eq_improved} from the improved Sobolev inequality in  \eqref{refsobolev} for $X=\dot{B}^{-N/2^\ast}_{\infty,\infty}$.
In order to do this, it will suffice to show that the Morrey spaces $\mathcal{L}^{1,\alpha}$ can be embedded in the Besov space $\dot{B}^{-\alpha}_{\infty, \infty}$ for any admissible $\alpha$.
\begin{lemma}\label{lem_10}
For any $\alpha \in (0,N)$ we have $\mathcal{L}^{1,\alpha} \hookrightarrow \dot{B}^{-\alpha}_{\infty, \infty}$.  More precisely, each function $u \in \mathcal{L}^{1,\alpha}$ is a tempered distribution and there exists a constant $C=C(N,\alpha)>0$
 such that
 \begin{equation}\label{eq_10star}
 \|u\|_{\dot{B}^{-\alpha}_{\infty,\infty}} \, \leq \, C \| u \|_{\mathcal{L}^{1,\alpha}},
 \end{equation}
 for any $u \in \mathcal{L}^{1,\alpha}(\R^N)$.
 \end{lemma}
 \begin{proof}
 By definition, since $u \in \mathcal{L}^{1,\alpha}$ then for any $\sigma>0$  and any ball $B_\sigma \subset \R^N$ the following estimate holds
 \begin{equation}\label{eq_11stella}
 \dys
 \int_{B_\sigma} |u(y)|\, dy \,  \leq \, \|u\|_{\mathcal{L}^{1,\alpha}} \sigma^{N-\alpha} \, .
 \end{equation}
 First, we are going to show that, because of \eqref{eq_11stella}, $u$ is  a tempered distribution. Then, taking $R=\sqrt{t}$ in definition~\eqref{def_besovnorm} it will remain to check that
 \begin{equation}\label{eq_10star2}
 \sup_{R>0} R^\alpha \| e^{R^2\Delta}u \|_{L^\infty} 
 \, \leq
 \, c \sup_{R>0, \, x\in R^N} R^\alpha \dashint_{B_R} |u|\, dy,
 \end{equation}
 for any $u\in \mathcal{L}^{1,\alpha}(\R^N)$.

The main tool will be a dyadic decomposition in the domain of integration. To this aim, for any fixed $R>0$, \,$x\in \R^N$ and $k\in \Z$, we define
 $$
 \cc_k := \Big\{ 2^{-k}R < |x-y| < 2^{-k+1}R; \ y\in \R^N \Big\} \, .
 $$

In order to prove that $u$ defines a tempered distribution, it is clearly enough to show that 
\begin{equation}
\label{tempd}
\int_{\R^N} |u(y)| |\phi(y)| dy \leq C \|u\|_{\mathcal{L}^{1,\alpha}} \, \sup_{y} |\phi(y)|(1+|y|^{m})
\end{equation}
for some integer $m \geq 1$, some absolute constant $C>0$ and for all $\phi \in \mathcal{S}(\R^N)$.

Since
\begin{equation}
\label{tempd1}
\int_{\R^N} |u(y)| |\phi(y)| dy\leq \left( \sup_{y} |\phi(y)|(1+|y|^{m}) \right) \int_{\R^N}\frac{|u(y)|}{1+|y|^m}dy \, ,
\end{equation}
we just need to estimate the last integral. 

Assuming $x=0$ and combining the dyadic decomposition  
 with  \eqref{eq_11stella} we have
\begin{eqnarray}\label{tempd2}
 \int_{\R^N}\frac{|u(y)|}{1+|y|^m}dy
&= & \sum_{k \in \mathds{Z}} \int_{\cc_k}\frac{|u(y)|}{1+|y|^m}dy 
\nonumber \\
 &\leq & \sum_{k \in \mathds{Z}}\frac{C}{1+2^{-mk}} 
  \int_{B_{2^{-k+1}}} |u(y)| dy \leq C \|u\|_{\mathcal{L}^{1,\alpha}} \sum_{k \in \mathds{Z}}\frac{2^{k(\alpha-N)}}{1+2^{-mk}} \, ,
  \end{eqnarray}
where the last sum is clearly convergent for any fixed $m\geq N$.
Combining \eqref{tempd1} with \eqref{tempd2} we see that \eqref{tempd} holds, hence $u\in \mathcal{S}^\prime$ as claimed.

To finish the proof we note that
 $$
 \|  e^{R^2\Delta}u\|_{L^\infty}
 \,  \leq \,  \|  e^{R^2\Delta}|u| \|_{L^\infty} \, ,
 $$
 therefore in proving \eqref{eq_10star2} we may assume that $u$ is nonnegative.
 
 \vspace{1mm}
 
 For $x\in \R^N$ and $R>0$ we set 
 $$
 K_R (x-y):=   \frac{1}{(4\pi)^{\frac{N}{2}}} \frac{1}{R^N} e^{-\frac{|x-y|^2}{4R^2}} \, , \qquad   G_R(x):= e^{R^2\Delta}u \, (x)=\int_{\R^N}  K_R(x-y) u(y)\, dy \, ,
 $$
 thus, our goal is to prove that $G_R(x)  \leq \, c \|u\|_{\mathcal{L}^{1,\alpha}} R^{-\alpha}
 $
 for any $x\in \R^N$ and $R>0$ for an absolute constant $c>0$.
  
  Using again the dyadic decomposition we can write
 \begin{equation}\label{eq_12stella}
 \dys
 G_R(x) = \sum_{k\in\Z} \int_{C_k}K_R(x-y)u(y)\, dy \, .
 \end{equation}
 Since on each annulus 
 we have $|x-y| \sim 2^{-k} R$, then we easily obtain 
\begin{eqnarray*}
\int_{\cc_{k}} K_R(x-y) u(y)\, dy
 & \leq  & c \frac{1}{R^N} e^{-c2^{-2k}} \int_{B_{2^{-k+1}R}} u(y)\,dy \\[1ex]
 & \leq & c   \|u\|_{\mathcal{L}^{1,\alpha}} R^{-N} e^{-c2^{-2k}} (2^{-k}R)^{N-\alpha},
 \end{eqnarray*}
 where we also used~\eqref{eq_11stella} by choosing $\sigma=\sigma_k\sim 2^{-k}R$ there. Now, we combine the inequality above with~\eqref{eq_12stella} 
  and we obtain
 \begin{equation}\label{eq_13stella2}
\dys
G_R(x) \leq c   \|u\|_{\mathcal{L}^{1,\alpha}}\left(\sum_{k\in\Z} 2^{-k(N-\alpha)} e^{-c2^{-2k}}\right)R^{-\alpha}
\, \leq \, c  \|u\|_{\mathcal{L}^{1,\alpha}}  R^{-\alpha},
\end{equation}
where we used that the sum in~\eqref{eq_13stella2} converges. 
Since the previous inequality holds for any $x \in \R^N$, we conclude that \eqref{eq_10star} holds, so the proof is complete.
 \end{proof}

 Now, we are in the position to provide another proof of the improved Sobolev inequality in the Morrey scale given by Theorem~\ref{thm_morrey1}.
 \begin{proof}[\bf Proof of Theorem~\ref{thm_morrey1}]
 In view of the previous results in this section,  when $\theta=2/2^\ast$ the inequality in~\eqref{eq_improved} is straightforward, by just combining the embedding \eqref{refsobolev} for $X=\dot{B}^{-N/2^\ast}_{\infty,\infty}$ proved in \cite{GMO97} with Lemmas~\ref{lem_10} and inequality \eqref{lem_8} (choosing $\alpha=(N-2s)/2$ and $\gamma=r(N-2s)/2$ respectively). Then the conclusion follows for the whole range $\theta \in [2/2^\ast, 1)$ just because
 $\dot{H}^s \hookrightarrow L^{2^\ast} \hookrightarrow \mathcal{L}^{r,r(N-2s)/2}$.
 \end{proof}

\vspace{2mm}

Analogously, for the case $s=1$ and $p\neq 2$, we have 

\begin{proof}[\bf Proof of Theorem~\ref{thm_morrey2}]
We can combine the results in this section with the refined Sobolev embedding for functions in $\dot{W}^{1,p}$ proved by Ledoux in~\cite{Led03}, that is
$$
\|u\|_{L^{p^\ast}}  \leq \, C \, \|u\|_{\dot{W}^{1,p}}^\theta \|u\|^{1-\theta}_{\dot{B}^{-(N-p)/p}_{\infty,\infty}}\,,
$$
where $p/p^\ast \leq \theta<1$, with Lemmas~\ref{lem_10} and inequality \eqref{lem_8} (choosing $\alpha=N/p^\ast$  and  $\gamma=r \alpha$ respectively) and we plainly obtain an alternative proof of~Theorem~\ref{thm_morrey2}.
\end{proof}
\vspace{1mm}

\section{Optimizing sequences for the Sobolev inequality}\label{sec_maxi}

As already mentioned in the Introduction, a first application of the improved Sobolev inequalities is the following lemma, which states that an appropriate scaling $\{x_n, \lambda_n \}$ will assure a nontrivial weak-limit 
of any sequence $\{u_n\}$ in $\dot{H}^s$ uniformly bounded from below in the Lebesgue $L^{2^\ast}$-norm.  
\begin{lemma}\label{lem_nontrivial}
Let $0 < s < N/2$ and let $\{ u_n\} \subset \dot{H}^s(\R^N)$ a bounded sequence such that
\begin{equation}\label{eq_25stella}
\dys
\inf_{n\in \N} \|u_n\|_{L^{2^\ast}} \geq c > 0.
\end{equation}
Then, up to subsequences, there exist a family of points $\{x_n\}\subset \R^N$ and a family of positive numbers $\{\lambda_n\}\subset (0,\infty)$ such that
$$
\tilde{u}_n \rightharpoonup w \neq 0 \quad \text{in} \ \dot{H}^s(\R^N),
$$
where $\tilde{u}_n(x):=\lambda_n^{\frac{N-2s}{2}}u_{n}\left(x_n+\lambda_nx  \right)$.
\end{lemma}

\begin{proof}
Since the sequence is bounded in $\dot{H}^s$ and $\dot{H}^s \hookrightarrow L^{2^\ast} \hookrightarrow \mathcal{L}^{2,N-2s}$, we have $\|u_n\|_{\mathcal{L}^{2,N-2s}} \leq C$ for some $C>0$ independent of $n$. On the other hand, combining assumption \eqref{eq_25stella} with Theorem \ref{thm_morrey1} for $r=2$, we readily see that $\|u_n\|_{\mathcal{L}^{2,N-2s}} \geq \tilde{C}$ for some $\tilde{C}>0$ independent of $n$. Hence, there exists a positive constant which we denote again by $C$ such that for any $n$ we have 
\begin{equation}\label{eq_26stella2}
\dys
C \, \leq \, \|u_n\|_{\mathcal{L}^{2,N-2s}} \, \leq \, C^{-1} \, .
\end{equation}

Combining the definition~\eqref{def_morrey} with \eqref{eq_26stella2}, we deduce that for any $n\in \N$ there exist $x_n, \lambda_n$ such that
\begin{equation}\label{eq_27stella}
\lambda_n^{-2s}\int_{B_{\lambda_n}(x_n)}|u_n(y)|^2 dy \, \geq \, \|u_n\|^2_{\mathcal{L}^{2,N-2s}} - \frac{C^2}{2n}
\, \geq \, \tilde{C}>0,
\end{equation}
for any $n$ and for some constant $\tilde{C}>0$.

Now we set $\tilde{u}_n(x):=\lambda_n^{\frac{N-2s}{2}}u_{n}\left(x_n+\lambda_n x \right)$ and in view of the scaling invariance of the $\dot{H}^s$ norm, the sequence $\tilde{u}_n$ is bounded in $\dot{H}^s$, therefore up to subsequences  $\tilde{u}_n \rightharpoonup w  \quad \text{in} \ \dot{H}^s(\R^N)$ and it remains to prove that $w \neq 0$.

Starting from \eqref{eq_27stella}, a direct calculation gives
$$
\int_{B_1} |\tilde{u}_n(x)|^2 dx = \lambda_n^{-2s}\int_{B_{\lambda_n}(x_n)}|u_n(y)|^2 dy \geq \tilde{C} >0 \, , 
$$
and since the embedding $\dot{H}^s \hookrightarrow L^2_{\rm{loc}}$ is compact, passing to the limit in the previous inequality we obtain $\int_{B_1}|w(x)|^2 dx \geq \tilde{C}>0$. Thus $w \neq 0$ and the proof is complete.
\end{proof}

We would like to remark that the improved Sobolev inequality in Morrey spaces $\mathcal{L}^{r, r \frac{N-2s}{2}}$ could be replaced by the analougue \eqref{refsobolev} with $X=\dot{B}^{s-N/2}_{\infty,\infty}$, with only minor changes in the present proof.

\vspace{1mm}

Now, combining Lemma~\ref{lem_nontrivial} together with the result in Brezis-Lieb (\cite{brezislieb}), 
 one can finally deduce that the optimal sequences in the critical Sobolev inequality~\eqref{eq_sobolev} are compact up to translations and dilations. 
 
\vspace{2mm}

\begin{proof}[\bf Proof of Theorem \ref{the_optimizseq}]
Let $\{ u_n\} \subset \dot{H}^s(\mathds{R}^N)$ be an optimal sequence for the Sobolev inequality \eqref{pbsobolev}, i.~\!e. a sequence such that $\| u_n \|_{\dot{H}^s}=1$ for each $n$ and $\int |u_n|^{2^{\ast}}dx \to S^{\ast}$ as $n\to \infty$. We aim to show that there exists a suitably rescaled subsequence converging strongly in $\dot{H}^s$ to a function $w \in \dot{H}^s(\mathds{R}^N)$. 

In view of Lemma~\ref{lem_nontrivial} applied to the optimal sequences $\{u_n\} \subset \dot{H}^s(\R^N)$, we have that the renumbered rescaled sequence $\tilde{u}_{n} \rightharpoonup w \neq 0$ in $\dot{H}^s(\R^N)$. It remains to show that $\tilde{u}_n \to w$ strongly in $\dot{H}^s$ and that $w$ is given by \eqref{def_talentiana}. This fact will follow from the optimality of the sequence $\{u_n\}$.

\vspace{1mm}

Note that the weak convergence $\tilde{u}_n\rightharpoonup w$ in $\dot{H}^s$ yields the identity
\begin{eqnarray}\label{eq_adriano}
&& \int_{\R^N} |(-\Delta)^{\frac{s}{2}} w|^2 dx +\limsup_{n \to\infty } \int_{\R^N} |(-\Delta)^{\frac{s}{2}} (\tilde{u}_n-w) |^2 dx \nonumber  \\
&& \qquad \qquad\qquad \qquad\qquad \qquad = \limsup_{n\to \infty}  \int_{\R^N} |(-\Delta)^{\frac{s}{2}} \tilde{u}_n|^2 dx =1. 
\end{eqnarray}
Since the embedding $\dot{H}^s\hookrightarrow L^2_{\rm{loc}}$ is compact, passing to a further subsequence if necessary we may also assume $\tilde{u}_n \to w$ a.~\!e..

Combining \cite{brezislieb} with Sobolev inequality \eqref{eq_sobolev}, the identity in~\eqref{eq_adriano} and the elementary inequality for positive numbers $a^{2^\ast/2}+b^{2^\ast/2} \leq (a+b)^{2^\ast/2}$, we have
\begin{eqnarray*}
S^{\ast} \! & = & \! \lim_{n \to\infty} \int_{\R^N} |\tilde{u}_n|^{2^{\ast}}dx \ = \ 
 \int_{\R^N} |w|^{2^{\ast}}dx+\lim_{n \to \infty}\int_{\R^N} |\tilde{u}_n-w|^{2^{\ast}}dx \\
\\
& \leq & \! S^{\ast} \left( \int_{\R^N} |(-\Delta)^{\frac{s}{2}} w|^2 dx \right)^{\!\frac{2^{\ast}}{2}}+S^{\ast} \left( \limsup_{n \to\infty }\int_{\R^N} |(-\Delta)^{\frac{s}{2}} (\tilde{u}_n-w) |^2 dx \right)^{\!\frac{2^{\ast}}{2}} \\
\\
& \leq & \! S^{\ast} \left( \int_{\R^N} |(-\Delta)^{\frac{s}{2}} w|^2 dx +\limsup_{n \to\infty } \int_{\R^N} |(-\Delta)^{\frac{s}{2}} (\tilde{u}_n-w) |^2 dx \right)^{\!\frac{2^{\ast}}{2}}  = \ S^{\ast}.
\end{eqnarray*}
Since all the previous inequalities are equalities, we infer  $\| w\|_{\dot{H}^s}=1$, because $w\not\equiv 0$ and hence $\tilde{u}_n \to w$ in $\dot{H}^s$. By Sobolev embedding \eqref{eq_sobolev} we also deduce $\tilde{u}_n \to w$ in $L^{2^{\ast}}$, hence $\int|w|^{2^{\ast}}dx=S^{\ast}$ and $w$ is an optimal function in the Sobolev inequality~\eqref{pbsobolev}, thus $w$ is given by \eqref{def_talentiana} in view of \cite[Theorem 1.1]{cotsiolis}.
\end{proof}

\vspace{1mm}

\section{Profile decomposition for arbitrary sequences in $\dot{H}^s$}\label{sec_decomposition}

In this section, we investigate the  {\it profile decomposition} for sequences of functions in the fractional Sobolev spaces~$\dot{H}^s(\R^N)$ by means of a careful analysis via  weak-convergence with respect to dilation and translation (the so-called $D$-weak convergence, where $D$ stands for ``dislocations''; see below). 
As mentioned in the introduction, we follow an abstract approach mainly due to Tintarev, contained in the book~\cite[Chapter 3]{TF07}.
 Also, in \cite[Chapters 5 and 6]{TF07}
 it has been also discussed the application to Sobolev spaces $\dot{H}^s(\R^N)$ when $s$ is an integer.
 Here, thanks to the improved Sobolev inequalities proved in the Section~\ref{sec_improved}, 
 we show how this abstract point of view can be easily applied
to the full range of  fractional Sobolev spaces~$\dot{H}^s(\R^N)$ for any real $0 < s < N/2$. Our presentation here differs slightly from \cite{TF07}: in particular we identify $D$ with the group $G$ of affine homogeneous dilation in $\mathds{R}^N$ and we take advantage of the corresponding explicit group structure. In our opinion such aspect, not present in \cite{Ger98} and not clearly used in \cite[Chapter 5]{TF07}, could be relevant for other concrete situations when different groups $G$ act on some function spaces with Hilbertian structure.

\vspace{2mm}

We start with a direct consequence of Lemma~\ref{lem_nontrivial}, which states that 
a characterization 
of the weak-convergence up to dilation and translation in $\dot{H}^s(\R^N)$ is precisely $L^{2^\ast}\!$-convergence. This fact is in clear accordance with the local case  $s=1$, as already proved in \cite[Lemma 5.3]{TF07}
and it shows that absence of profiles can be measured in the $L^{2^\ast}\!$-norm.
 We have

\begin{prop}\label{prop_sdconv}
Let $\{u_n\}$ be any bounded sequence in $\dot{H}^s(\R^N)$, then the following statements are equivalent
\begin{itemize}
\item[{\rm (i)}]{
For any $\{x_n\} \subset \R^N$ and any $\{\lambda_n\} \subset (0,\infty)$,
$$
\dys
\tilde{u}_n (\cdot) := \lambda_n^{\frac{N-2s}{2}}u_{n}(x_n+\lambda_n\,\cdot) \rightharpoonup \, 0 \ \, \text{in} \ \dot{H}^s(\R^N) \  \, \text{as} \ {n\to \infty}.
$$
}
\item[{\rm(ii)}]{
$\{ u_n \}$ converges strongly to $0$ in $L^{2^\ast}\!(\R^N)$. 
}
\end{itemize}
\end{prop}
\begin{proof}
The first implication is now straightforward thanks to Lemma~\ref{lem_nontrivial}. Indeed, assume by contradiction that $u_n \nrightarrow 0$ in $L^{2^\ast}(\R^N)$, then we have that, up to subsequences, there exist $\bar{x}_n$ and $\bar{\lambda}_n$ such that $\bar{\lambda}_n^{\frac{N-2s}{2}} u_{n}(\bar{x}_n +\bar{\lambda}_n\,\cdot)$ weakly converges to a function $w\neq 0$. Contradiction.
\vspace{2mm}

Assume now that $u_n\rightarrow 0$ strongly in $L^{2^\ast}(\R^N)$ as $n\to \infty$. Then, for any $\{x_n\}$ and $\{\lambda_n\}$, by scaling invariance, it also holds $\|\tilde{u}_n\|_{L^{2^\ast}} \rightarrow 0$ as $n\to \infty$, so that, by continuity of the Sobolev embedding~\eqref{eq_sobolev}, it follows $\tilde{u}_n \rightharpoonup 0$ in $\dot{H}^s(\R^N)$ as $n\to \infty$, as desired. 
\end{proof}

The following result shows that the previous proposition can be improved using a Morrey norm weaker than $L^{2^\ast}$ to characterize absence of profiles.

\begin{corollary}
\label{nofprofiles}
Let $\{u_n\}$ be any bounded sequence in $\dot{H}^s(\R^N)$. For fixed $1\leq r < 2^\ast$ the following statements are equivalent
\begin{itemize}
\item[{\rm (i)}]{
For any $\{x_n\} \subset \R^N$ and any $\{\lambda_n\} \subset (0,\infty)$,
$$
\dys
\tilde{u}_n(\cdot) := \lambda_n^{\frac{N-2s}{2}}u_{n}(x_n+\lambda_n\,\cdot) \rightharpoonup \, 0 \ \, \text{in} \ \dot{H}^s(\R^N) \  \, \text{as} \ {n\to \infty}.
$$
}
\item[{\rm(ii)}]{
$\{ u_n \}$ converges strongly to $0$ in $\mathcal{L}^{r, r\frac{N-2s}{2}}(\R^N)$ as $n \to \infty$. 
}
\end{itemize}
\end{corollary}
\begin{proof}
Clearly (i) implies (ii) as a consequence of the previous proposition and \eqref{eq_6star}.  Conversely, if (ii) holds, then $u_n \to 0$ in $L^{2^\ast}$ in view of Theorem \ref{thm_morrey1}. Then, the conclusion  again follows from the previous proposition.
\end{proof}

Actually, the statement in~$\rm (i)$ in Proposition \ref{prop_sdconv} above is the definition of the 
{\it $D$-weak convergence in $\dot{H}^s(\R^N)$} to zero. Whereas we will not go into detail of the $D$-weak convergence framework for general Hilbert space (referring the interested readers to \cite[Chapter 3]{TF07}), in the following lines we will describe how to equip the fractional Sobolev spaces~$\dot{H}^s(\R^N)$ with an appropriate group of {\it dislocations} $D$. This step is necessary in order to apply the abstract profile decomposition, namely the Refined Banach-Alaoglu Theorem on dislocation spaces~\cite[Theorem 3.1]{TF07}, in turn implying the desired profile decomposition in $\dot{H}^s(\R^N)$.

\vspace{2mm}

Let $T \subset \mathcal{U}(\dot{H}^s(\R^N))$ the group of all unitary operator on $\dot{H}^s(\R^N)$ induced by translation on $\mathds{R}^N$, i.~\!e.
$$
\dys
T:= \Big\{ T_y   \, : \,  T_y u(x):=u(x-y)  \quad  \forall \, u \in  \dot{H}^s \, ; \, y\in \R^N \,  \Big\},
$$
and by $S \subset \mathcal{U}(\dot{H}^s(\R^N))$ the group of all unitary operator on $\dot{H}^s(\R^N)$ induced by dilations on $\mathds{R}^N$, i.~\!e. 
$$
\dys
S:= \Big\{ S_\lambda  \, : \,  S_\lambda u(x):= \lambda^{\frac{2s-N}{2}}u( \lambda^{-1}\, x) \quad  \forall \, u \in  \dot{H}^s \, ; \,  \lambda>0 \,  \Big \}.
$$
Now, we consider the family of operators $D$ given by the composition of elements 
 of the preceding groups (actually the semidirect product of $T$ and $S$); i.~\!e., we define $D \subset \mathcal{U}(\dot{H}^s(\R^N))$ as 
 \begin{equation}\label{def_di}
 D:= \Big \{ D_{y,\lambda} \, : \,  D_{y,\lambda}u(x):=\lambda^{\frac{2s-N}{2}} u\left(\frac{x-y}{\lambda}\right) \quad  \forall \, u \in  \dot{H}^s \, ; \, y\in \R^N \, , \lambda>0 \, \Big\}.
 \end{equation}
Notice that the map 
$\R^N \times(0,\infty)\to \mathcal{U}(\dot{H}^s(\R^N))$ defined above
is continuous in the strong topology of $\mathcal{U}(\dot{H}^s(\R^N))$; we have that, for any sequence~$(y_n, \lambda_n)$ converging to $(y,\lambda)$, 
\begin{equation}\label{eq_34stella}
\dys
D_{y_n,\lambda_n}u \rightarrow D_{y,\lambda}u   \quad \forall u \in \dot{H}^s(\R^N) \ \text{as} \ n\to \infty.
\end{equation}
This can be checked on smooth functions $u\in C^\infty_0(\R^N)$, via Fourier transform, using the Dominated Convergence Theorem and then arguing by density.

\vspace{1.5mm}

Now, let us focus on $G= \R^N\rtimes (0,\infty)$ as a semi-direct product of groups $\mathds{R}^N$ and~$(0,\infty)$. As a set, $G=\R^N\times (0,\infty)$ is naturally endowed with the distance 
\begin{equation}
\label{distance}
d((x,\lambda), (y,\sigma))=\left| \log \frac{\lambda}{\sigma}\right| +| x-y | \,  \, , 
\end{equation}
which makes it a complete metric space. It is easy to check that it induces the usual Euclidean topology with the usual notion of convergent sequences.
It is very convenient to identify $G$
with the set of homogeneous affine transformations acting on $\mathds{R}^N$, i.~\!e. 
$$
\dys
\R^N\rtimes(0,\infty) \ni (y,\lambda) \longleftrightarrow
\varphi_{y,\lambda}: x\mapsto \frac{x-y}{\lambda} \, , \quad \varphi_{y,\lambda} \in {\it Aff} (\mathds{R}^N) \, ,
$$
whence $G$ inherits a group structure just by composition of the corresponding affine transformation; more precisely, since
$\varphi_{a,\delta}\circ \varphi_{y,\lambda}(x) = {(x-(y+a\lambda))}/{\delta\lambda}$, and therefore,
\begin{equation}\label{eq_40stella}
\dys
\varphi_{a,\delta}\circ \varphi_{y,\lambda} \equiv 
\varphi_{y+a\lambda,\delta\lambda}  \ \ \forall a,y\in \R^N, \ \forall \delta,\lambda \in (0,\infty),
\end{equation}
then we define the group law on $G=\mathds{R}^N \rtimes (0,\infty)$ by setting
\begin{equation}
\label{grouplaw}
(y,\lambda) \circ (a,\delta) := (y+\lambda a, \lambda \delta) \, , \qquad (y,\lambda), \, (a,\delta) \in G \, .
\end{equation}

Therefore, denoting by $\pi:\R^N\rtimes(0,\infty) \rightarrow \mathcal{U}(\dot{H}^s(\R^N))$ the following map
\begin{equation}\label{def_pi}
\pi(y,\lambda) = D_{y,\lambda},
\end{equation}
one can check that it is a group homomorphism (i.~\!e., unitary representation); also, $\pi$~is injective and strongly continuous in view of~\eqref{eq_34stella}. In this way, we have $D=\pi(G)$ where $G=\R^N\rtimes(0,\infty)$ and $D$ is given by~\eqref{def_di}.

\vspace{2mm}

The following results show how translations and dilations act on the space $\dot{H}^s$. In particular, we first characterize how a given function becomes asymptotically orthogonal in $\dot{H}^s$ with respect  to any fixed function under a sequences of scaling going to infinity on the group $G$. We have the following
\begin{lemma}\label{lemma1}
Take any sequence $\{(y_n,\lambda_n)\} \subset \R^N\times(0,\infty)$. Then the following statements are equivalent as  $n \to \infty$
\begin{itemize}
\item[(i)]{
$D_{y_n,\lambda_n} v\rightharpoonup 0  \, , \quad \forall  v \in \dot{H}^s(\R^N).$
}\vspace{1mm}
\item[(ii)]{ 
$|\log{\lambda_n}| + |y_n| \to \infty.$
}
\end{itemize}
\end{lemma}
\begin{proof}

Assume that the convergence in $\rm (i)$ does hold; thus 
\begin{equation}\label{eq_dconv0}
\langle w, D_{y_n,\lambda_n}v\rangle_{\dot{H}^s} \, \underset{n\to\infty}\longrightarrow \, 0 \quad \forall w, v \in \dot{H}^s(\R^N).
\end{equation}
We argue by contradiction. If ${\rm (ii)}$ fails, then up to subsequences $y_n\to y\in \R^N$ and $\lambda_n\to \lambda>0$ as $n$ goes to infinity. Thus, for any $u\neq 0$, by continuity of the scalar product and \eqref{eq_34stella}, one can write
$$
\lim_{n\to \infty} \langle D_{y_n,\lambda_n}u, D_{y_n,\lambda_n}u\rangle_{\dot{H}^s} \, = \, 0,
$$
which, still in view of \eqref{eq_34stella} and the fact that $D \subset \mathcal{U}(\dot{H}^s)$ also gives  
$$
0 \, = \, \langle D_{y,\lambda}u, D_{y,\lambda}u\rangle_{\dot{H}^s} 
\, = \, \|D_{y,\lambda}u \|^2_{\dot{H}^s} 
\, \equiv \, \|u\|^2_{\dot{H}^s},
$$
and this is a contradiction, since we have taken $u\neq 0$.

\vspace{2mm}
Conversely, now we assume that the condition~$\rm (ii)$ does hold. Thus, up to renumbered subsequence, we have to distinguish between the following two cases: $\lambda_n\to 0$ or $\lambda_n \geq \lambda > 0$.
\vspace{1mm}

It is clearly enough to check the property ${\rm (i)}$ on the dense space of all~$u, v$ in the Schwarz class of rapidly decaying functions $\mathcal{S}(\R^N)$ such that $\mathcal{F}(u) \in C^\infty_0(\R^N\setminus\{0\})$. Notice that, under this assumption, we have $(-\Delta)^s u, (-\Delta)^s v \in \mathcal{S}(R^N)$.

\vspace{2mm}

We can write
\begin{eqnarray}\label{eq_38stella}
\dys
\langle u, D_{y_n,\lambda_n} v \rangle_{\dot{H}^s} & = & \int (-\Delta)^{\frac{s}{2}} u\, (-\Delta)^{{\frac{s}{2}}} (D_{y_n,\lambda_n} v) \, dx \nonumber \\
& = & \int ((-\Delta)^{s} u)\, D_{y_n,\lambda_n} v \,dx, \quad \forall u, v \in \mathcal{S}.
\end{eqnarray}
Therefore,  it suffices to analyze the convergence in the identity above, with respect to the aforementioned two cases.

\vspace{2mm}

In the first case,  that is when $\lambda_n \geq \lambda>0$, we have that $D_{y_n,\lambda_n}$ is equi-bounded in $L^\infty$. Moreover, since $|y_n|\to \infty$ and/or $\lambda_n\to \infty$, we have that $D_{y_n,\lambda_n} v \to 0$ a.~\!e. as $n\to \infty$, so that the Dominated Convergence Theorem yields~\eqref{eq_dconv0}, i.~\!e. (i) holds.

\vspace{2mm}
In the second case,  that is when $\lambda_n \to 0$, we can write
$$
\dys
\langle u, D_{y_n,\lambda_n} v \rangle_{\dot{H}^s} \, = \,  \int (-\Delta)^s u(x+y_n) \,\lambda^{-\frac{N-2s}{2}}v\left(\frac{x}{\lambda_n}\right)\, dx. 
$$
For any $\delta>0$ it is convenient to split the integral in two parts, respectively
\begin{eqnarray}\label{levico}
\dys
&& \int_{\R^N} (-\Delta)^s u(x+y_n) \,\lambda^{-\frac{N-2s}{2}}v\left(\frac{x}{\lambda_n}\right)\, dx \nonumber \\
&&\qquad\qquad\qquad = \, \int_{\cc B_\delta(0)} (-\Delta)^s u(x+y_n) \,\lambda^{-\frac{N-2s}{2}}v\left(\frac{x}{\lambda_n}\right)\, dx \nonumber \\
& &\qquad\qquad\qquad \quad + \, \int_{B_\delta(0)} (-\Delta)^s u(x+y_n) \,\lambda^{-\frac{N-2s}{2}}v\left(\frac{x}{\lambda_n}\right)\, dx,
\end{eqnarray}
where $\cc B_{\delta}(0)$ denotes the complement of the ball of radius $\delta$ centered in 0.

The first integral in the right-hand side of~\eqref{levico} can be estimated as follows
\begin{eqnarray}\label{lev2}
\dys
&& \left|
\int_{\cc B_\delta(0)} (-\Delta)^s u(x+y_n) \,\lambda^{-\frac{N-2s}{2}}v\left(\frac{x}{\lambda_n}\right)\, dx
\right| \nonumber \\
&& \qquad\qquad\qquad\, \leq \,
\left(\int_{\R^N} |(-\Delta)^s u| \, dx \right) \|D_{y_n,\lambda_n} v\|_{L^\infty(\cc B_\delta(y_n))} 
\, \overset{n\to\infty}\longrightarrow \, 0.
\end{eqnarray}

For the second integral, we can use the H\"older inequality to get, for any $\delta >0$,
\begin{eqnarray}\label{lev3}
\dys
&& \left| \int_{B_\delta(0)} (-\Delta)^s u(x+y_n) \,\lambda^{-\frac{N-2s}{2}}v\left(\frac{x}{\lambda_n}\right)\, dx \right| \nonumber \\
&&\qquad\qquad\qquad \leq \, \|(-\Delta)^s u\|_{L^\infty(\R^N)} \, \int_{B_\delta(0)} \left| \lambda^{-\frac{N-2s}{2}} v\left(\frac{x}{\lambda_n}\right)\right|\, dx \nonumber \\
&&\qquad\qquad\qquad \leq \, \|(-\Delta)^s u\|_{L^\infty(\R^N)} \left( \int_{B_\delta(0)} \left| \lambda^{-\frac{N-2s}{2}} v\left(\frac{x}{\lambda_n}\right)\right|^{2^\ast}\, dx \right)^{\frac{1}{2^\ast}} |B_\delta(0)|^{\frac{2N}{N+2s}} \nonumber \\
&&\qquad\qquad\qquad \leq \, c\|v\|_{L^{2^\ast}(\R^N)} \delta^{\frac{2N}{N+2s}},
\end{eqnarray}
where we also used the scaling invariance of the $L^{2^\ast}$-norm.
Finally, combining~\eqref{levico} with \eqref{lev2} and \eqref{lev3}, we have
$$
\limsup_{n\to\infty} \left| \int_{\R^N} (-\Delta)^s u(x+y_n) \,\lambda^{-\frac{N-2s}{2}}v\left(\frac{x}{\lambda_n}\right)\, dx  \right| 
\, \leq \, c \|v\|_{L^{2^\ast}}\delta^{\frac{2N}{N+2s}},
$$
which yields~\eqref{eq_dconv0} as $\delta\to 0$. Thus condition (i) holds also in this second case and the proof is complete.
\end{proof}

Taking into account the group properties~\eqref{eq_40stella}-\eqref{def_pi}, we can write
$$
\dys 
D^{-1}_{a_n,\delta_n} \circ D_{{y_n},{\lambda_n}} 
\, = \pi ( (a_n,\delta_n)^{-1} \circ (y_n, \lambda_n)) = \, D_{\delta_n^{-1}(y_n-a_n),\delta_n^{-1} \lambda_n}  \, .
$$

Combining Lemma~\ref{lemma1} with the previous identity we obtain the following result, which, roughly speaking, say that two arbitrary $\dot{H}^s$-functions are made orthogonal by sequences of elements of the group going at infinity in ``different directions''. As a consequence we will see that asymptotic orthogonality is obtained either if the dilation parameters are not comparable or if they are comparable each other but negligible with respect  to the distance between the translation parameters as $n \to \infty$.
\begin{lemma}\label{lemma2}
For any sequences $\{(a_n,\delta_n)\}, \{(b_n,\lambda_n)\} \subset \R^N\times(0,\infty)$. Then the following statements are equivalent as  $n \to \infty$
\begin{itemize}
\item[(i)]{
$\dys \langle D_{a_n,\delta_n}u, D_{y_n,\lambda_n}v\rangle_{\dot{H}^s} \rightarrow 0 \quad \forall \, u, v \in \dot{H}^s(\R^N).$
}\vspace{1mm}
\item[(ii)]{ 
$\dys \left|\frac{y_n -a_n}{\delta_n}\right| + \left|\log{\left(\frac{\lambda_n}{\delta_n}\right)}\right|\rightarrow \infty.$
}
\end{itemize}
\end{lemma}

Note that an equivalent version of the previous lemma is already present in \cite{Ger98}; that result holds in $L^2(\mathds{R}^N)$ and the notion of ``different directions'' given in (ii) here corresponds to the notion of orthogonality used there. The way translations and dilation acts on $L^2$ is obviously different but it still corresponds to a unitary representation of the same group $G$ on the Hilbert space $L^2$.  Of course the operator $(-\Delta)^{s/2}:\dot{H}^s \to L^2$ is unitary, it makes the two representations of $G$ unitary equivalent and it is easy to check that the two orthogonality property, in $L^2$ from \cite{Ger98} and in $\dot{H}^s$ here, coincide and always correspond to condition (ii) on the parameters. However, in our opinion the present approach is much simpler than the one based on a subtle analysis of $h-$oscillating sequences used in \cite{Ger98}.

\vskip2mm 
In view of the preceding lemmas, we can apply \cite[Proposition 3.1]{TF07}, and so we can conclude that the pair $(\dot{H}^s, D)$ is a {\it dislocation space} in the sense of~\cite[Definition 3.2 and Remark 3.1]{TF07}. Thus, the abstract theory by Tintarev can be applied,  
in turn implying the desired profile decomposition for $\dot{H}^s(\R^N)$, as given in the following theorem.
\begin{thm}\label{thm_TF}{\rm (}\cite[Theorem 3.1, Corollary 3.2]{TF07}{\rm )}.
Let $(H,D)$ be a dislocation space, where $D\subset\mathcal{U}(H)$. 
 If $\{ u_n \} \subset H$ is a bounded sequence, then there exist a (at most countable) set $J$, 
$\psi^{(j)} \in H$, 
$ {g}_n^{(j)} \in D$, ${g}^{(1)}_n = Id$, with $n \in \N$, $j\in J$ such that for a renumbered subequence, as $n \to \infty$ we have 
\begin{eqnarray*}
\dys
&& \psi^{(j)} = w -\lim  {g}_n^{(j)^\ast} u_n, \\[1ex]
&& {g}_n^{(j)^\ast}{g}_n^{(m)} \rightharpoonup 0 \ \text{for} \ j \neq m, \\[1ex]
&& \|u_n\|^2_H = \sum_{j\in J} \|\psi^{(j)} \|_H^2 +\|r_n||^2_H+o(1) \\[1ex]
&& u_n- \sum_{j\in J} {g}_n^{(j)}\psi^{(j)}\overset{D}\rightharpoonup 0,
\end{eqnarray*}
where the series $\sum_{j\in J} g_n^{(j)}\psi^{(j)}$ converges in $H$ uniformly in $n$.
\end{thm}

\vspace{2mm}

\begin{proof}[\bf Proof of Theorem~\ref{thm_decomposition}]
Taking into account Proposition \ref{prop_sdconv} and Lemma~\ref{lemma2} a-\break bove, the conclusion follows readily from Theorem~\ref{thm_TF}.  
\end{proof}

\vspace{1mm}

\section{Concentration-compactness Alternative}\label{sec_cca}

This section is devoted to the proof of Theorem~\ref{the_cca} and its consequences, i.~\!e. we establish the concentration-compactness alternative and we describe the behavior of the optimal sequences for the Sobolev inequality \eqref{pbsobolev2}.
 We  show that in the case of bounded domain there is no energy loss in the concentration process (see Proposition~\ref{pro_supporto}) and that the maximizing sequences for the Sobolev inequality always concentrate at one point (see Corollary~\ref{cor_concentrazio-sob}). We analyze the asymptotic behavior of the subcritical Sobolev constant $S^\ast_\eps$ and the corresponding optimal functions proving Theorem~\ref{the_concentrazione1}.

\vspace{2mm}
First we need some tools to handle the nonlocality of the fractional Laplacian. 

\subsection{Some useful lemmas}
 Roughly speaking, Lemma~\ref{lem_cutoff} and  Lemma~\ref{lem_commutator} below are workarounds to use cut-off functions and provide a way to manipulate smooth truncations for the fractional Laplacian; their proofs carefully requires properties of multipliers between Sobolev spaces and strong commutator estimates.

\begin{lemma}
\label{lem_cutoff}
Let $0<s<N/2$ and let  $u\in \dot{H}^s(\mathds{R}^N)$. Let $\varphi \in C^\infty_0(\mathds{R}^N)$ and for each $\lambda>0$ let $\varphi_\lambda (x):=\varphi(\lambda^{-1}x)$. Then
$$
u\varphi_\lambda \to 0 \ \text{in} \ \dot{H}^s(\mathds{R}^N) \ \text{as}\ \lambda \to 0.
$$
If, in addition, $\varphi\equiv 1$ in a neighborhood of the origin, then
$$
u\varphi_\lambda \to u  \ \text{in} \ \dot{H}^s(\mathds{R}^N) \ \text{as}  \ \lambda \to \infty.
$$
\end{lemma}
\begin{proof}
First, note that each function $\varphi_\lambda$ gives a bounded multiplication operator  $M_{\varphi_\lambda} \in \mathscr{L}(\dot{H}^s,\dot{H}^s)$ with operator norm independent on $\lambda$ because of the scale invariance of the $\dot{H}^s$ norm (see \cite[Chapter 3]{MS}, where instead of $\dot{H}^s$ the more traditional notation $h^s_2$ is used for the Riesz potential space of order $s$ and summability two).
\vspace{1mm}
\\ Thus, if $C\equiv \| \varphi_\lambda\|_{\mathcal{L}(\dot{H}^s,\dot{H}^s)}$ we have
\begin{equation}
\label{multest}
\| v \varphi_\lambda\|_{\dot{H}^s} \leq C \| v\|_{\dot{H}^s} 
\end{equation}
for any $v \in \dot{H}^s$.

By density we take a sequence $\{ u_n\} \subset C^\infty_0 (\mathds{R}^N)$ such that $u_n \to u$ in $\dot{H}^s$, so we can estimate
\begin{equation}
\label{zero}
\| u\varphi_\lambda \|_{\dot{H}^s} \, \leq \, \| (u-u_n)\varphi _\lambda\|_{\dot{H}^s}+\| u_n\varphi_\lambda\|_{\dot{H}^s} \, \leq \, C \| (u-u_n)\|_{\dot{H}^s}+\| u_n\varphi_\lambda\|_{\dot{H}^s}\, .
\end{equation}
Since for fixed $n$ the function $u_n$ gives also a bounded multiplier on $\dot{H}^s$, we have
\begin{equation}
\label{first}
\| u_n\varphi_\lambda\|_{\dot{H}^s} \leq C(u_n) \| \varphi_\lambda \|_{H^s}\to 0 
\end{equation}
as $\lambda \to 0$ by a direct scaling argument. Thus, the first statement of the lemma follows from \eqref{zero} and \eqref{first} letting $\lambda \to 0$ and $n \to \infty$.

\vspace{2mm}

In order to prove the second statement, it is enough to note  that whenever $u \in C^\infty_0(\mathds{R}^N)$ (indeed for any $u$ which is compactly supported) we have $u\varphi_\lambda \equiv u$ for $\lambda$ sufficiently large (depending on $u$). Thus we see that $u \varphi_\lambda \to u$ as $\lambda \to \infty$ for any $u \in C^\infty_0(\mathds{R}^N)$ and the same holds for any $u\in \dot{H}^s(\Om)$ by approximation.
\vspace{1mm}

Indeed,  \eqref{multest} gives
\begin{eqnarray}\label{second}
\| u- u\varphi_\lambda \|_{\dot{H}^s} & \leq & \| (u-u_n)(1-\varphi _\lambda)\|_{\dot{H}^s}+\| u_n(1-\varphi_\lambda)\|_{\dot{H}^s} \nonumber \\
& \leq & (1+C) \| (u-u_n)\|_{\dot{H}^s}+\| u_n(1-\varphi_\lambda)\|_{\dot{H}^s}\, ,
\end{eqnarray}
and the conclusion follows arguing as in the previous case.
\end{proof}
\vspace{2mm}

\begin{lemma}
\label{lem_commutator}
Let $0<s<N/2$, let $\Omega \subset \mathds{R}^N$ a bounded open set and let $\varphi~\!\in~\!C^\infty_0(\mathds{R}^N)$. Then the commutator $\, [\varphi, (-\Delta)^{\frac{s}{2}}]:\dot{H}^s(\Omega) \to L^2(\mathds{R}^N) \,$ is a compact operator, i.~\!e.
$$
\varphi  ((-\Delta )^{\frac{s}{2}}u_n)- (-\Delta)^{\frac{s}{2}} (\varphi u_n) \to 0 \quad \hbox{in} \quad L^2(\mathds{R}^N)
$$
whenever $u_n \rightharpoonup 0$ in $\dot{H}^s(\Omega)$ as $n \to \infty$.
\end{lemma}

\begin{proof}
Let $L=(-\Delta)^{\frac{s}{2}}$ and for each $\varepsilon>0$ set $L_\varepsilon=(\varepsilon Id-\Delta)^{\frac{s}{2}}$. Clearly, by conjugation with Fourier transform
we have 
$$
Lu=\mathcal{F}^{-1}\circ M_{|\xi|^s} \circ \mathcal{F} (u) \ \ \ \ \text{and} \ \ \ \ L_\varepsilon u=\mathcal{F}^{-1}\circ M_{(|\xi|^2+\varepsilon)^{\frac{s}{2}}} \circ \mathcal{F} (u)\,.
$$
Thus, $L_\varepsilon:H^s(\mathds{R}^N) \to L^2(\mathds{R}^N)$ is a bounded operator which in turn implies the boundedness of the operator $L_\varepsilon:\dot{H}^s(\Omega) \to L^2(\mathds{R}^N)$ induced by the continuous embedding $\dot{H}^s(\Omega) \hookrightarrow H^s(\mathds{R}^N)$.
\vspace{1mm}

Similarly, $L:H^s(\mathds{R}^N) \to L^2(\mathds{R}^N)$ is a bounded operator and the induced operator $L:\dot{H}^s(\Omega) \to L^2(\mathds{R}^N)$ is also bounded.

\vspace{1mm}

Estimating the norm in $\mathscr{L}(H^s, L^2)$ easily yields 
$$ 
\| L_\varepsilon-L\| \, \leq \, \sup_\xi \frac{|(\varepsilon+|\xi|^2)^{\frac{s}{2}}-|\xi|^s|}{(1+|\xi|^2)^\frac{s}{2}} \overset{\varepsilon \to 0}{\longrightarrow} 0 \, ,
$$
hence the same holds in $\mathscr{L}(\dot{H}^s(\Omega), L^2(\mathds{R}^N))$.

\vspace{2mm}

Thus, it suffices to prove that 
$$ 
[ L_\varepsilon, \varphi]:\dot{H}^s(\Omega) \to  L^2(\mathds{R}^N)
$$
is a compact operator for each $\varepsilon>0$, to deduce the same property for $[L,\varphi]$. 

\vspace{1mm}
 Let $L_\varepsilon=(\varepsilon Id-\Delta)^{\frac{s}{2}}$ and $l_\varepsilon(\xi)=(|\xi|^2+\varepsilon)^{\frac{s}{2}}$ the corresponding symbol. Clearly, $L_\varepsilon$ is a classical pseudodifferential operator of order $s$, i.~\!e. $L_\varepsilon \in OPS^s_{1,0}$ (hence $L_\varepsilon \in OP\mathcal{B}S^s_{1,1}$). Since $0<s<{N}/{2}$, according to \cite[Proposition 4.2]{t} we have the following commutator estimate
\[ \| [L_\varepsilon, \varphi] u\|_{L^2(\mathds{R}^N)} \leq C \| \varphi\|_{H^\sigma (\mathds{R}^N)} \| u\|_{H^{s-1}(\mathds{R}^N)} \, ,\]
provided $\sigma>{N}/{2}+1$.

\vspace{1mm}

Finally, as $\Omega$ is bounded, the embedding $\dot{H}^s(\Omega)~\!\hookrightarrow~\!H^{s-1}(\mathds{R}^N)$ is compact for all $s \in (0, {N}/{2})$ and $\varphi \in C^\infty_0(\mathds{R}^N)$, from the previous inequality we conclude that $[ L_\varepsilon, \varphi]:\dot{H}^s(\Omega) \to  L^2(\mathds{R}^N)$ is compact, as desired.
\end{proof}

\subsection{Concentration-compactness} 

We start with a well known lemma about pairs of positive measures in the Euclidean space. Roughly speaking, it gives control on their atomic parts whenever a reverse H\"older inequality holds.

\begin{lemma}\label{lem_reverse}{{\rm(}\rm\cite{lions})}
Let $\Om \subseteq \mathds{R}^N$ be an open set and let $\mu$ and $\nu$ in $\mathcal{M}(\mathds{R}^N)$ be two nonnegative bounded measures with support in $\overline{\Omega}$ such that for some $1\leq p < r < \infty$ there exists a positive constant $C$ such that
\begin{equation}\label{eq_reverse}
\dys \left(\int_{\mathds{R}^N}|\varphi|^rd\nu\right)^{\frac{1}{r}} \leq C \left(\int_{\mathds{R}^N}|\varphi|^pd\mu\right)^{\frac{1}{p}} \ \ \forall \varphi \in C^{0}_{0}(\mathds{R}^N).
\end{equation}\vspace{1mm}
\noindent
\\ Then, there exist a number $\sigma=C^{-(p^{-1}-r^{-1})^{-1}}>0$, a (at most countable) set of distinct points $\{x_j\}_{j\in J}$ in $\Omb$ and positive numbers $\nu_j \geq \sigma$, $j\in J$, such that
\begin{equation}\label{eq_precise}
\dys \nu=\sum_{j} \nu_j \delta_{x_j} \ \ \ \text{and} \ \ \ \mu\geq C^{-p}\sum_{j} \nu_j^{\frac{p}{r}}\delta_{x_j},
\end{equation}
where $\delta_{x_j}$ denotes the Dirac mass at $x_j$.  
\end{lemma}

Using the previous lemma we are able to prove the main result of this section, i.~\!e. Theorem \ref{the_cca}. Namely we show that the well-known concentration-compactness alternative holds for sequences in any Sobolev spaces $\dot{H}^{s}(\Om)$, $0<s<N/2$. The proof follows the original arguments in \cite{lions} and \cite{lions2} with some modifications to handle fractional differentiation.

\begin{proof}[\bf Proof of Theorem \ref{the_cca}]
Since $\dot{H}^s(\Omega) \hookrightarrow L^2_{\rm{loc}}(\mathds{R}^N)$ with compact embedding, passing to a subsequence if necessary, we may assume that $u_n \to u$ both in $L^2_{\rm{loc}}(\mathds{R}^N)$ and a.~\!e..
Similarly, for $v_n=u_n-u  \rightharpoonup 0$ in $\dot{H}^s(\Omega)$, up to subsequence, we may assume
$$
|(-\Delta)^{\frac{s}{2}} v_n|^2dx \tows \hat{\mu} \ \ \ \text{and} \ \ \ |v_n|^{2^{\ast}}dx\tows \hat{\nu} \ \ \text{in} \ \mea, 
$$
for some positive measures $\hat{\mu}$ and $\hat{\nu}$ with spt~$\!\hat{\nu} \subset \overline{\Omega}$.
In addition, when $\Omega$ is bounded, Lemma \ref{lem_commutator} easily yields spt~$\!\hat{\mu}\subset \overline{\Omega}$.

Clearly $\nu \geq |u|^{2^{\ast}}dx$ by Fatou's Lemma,  and combining pointwise convergence and the result in \cite{brezislieb} we have
\begin{eqnarray*}
\int_{\mathds{R}^N} |\varphi|^{2^{\ast}}d \nu- \int_{\mathds{R}^N} |\varphi u|^{2^{\ast}} dx \! & = & \! \lim_{n\to \infty} \int_{\mathds{R}^N} |\varphi u_n|^{2^{\ast}}dx - \int_{\mathds{R}^N} |\varphi u|^{2^{\ast}} dx \\
\\
& = & \! \lim_{n\to \infty} \int_{\mathds{R}^N} |\varphi v_n|^{2^{\ast}}dx
\ =  \ \int_{\mathds{R}^N} |\varphi|^{2^{\ast}} d\hat{\nu},
\end{eqnarray*}
i.~\!e. $\nu=\hat{\nu}+|u|^{2^{\ast}}dx$ because the function $\varphi \in C^0_0(\mathds{R}^N)$ can be choosen arbitrarily.
\vspace{2mm}

We are going to prove the structure properties in \eqref{quantnu} and  \eqref{quantmu} assuming that $\Omega$ is bounded. Then, the structure relation \eqref{quantnu} will be true for any $\Omega$ just by a simple localization argument.

Indeed, fix  $\psi \in C^\infty_0(\mathds{R}^N)$ such that $\psi \equiv 1$ on $B_1$ and for $0<\lambda<1$ let $\psi_\lambda(x)~=~\psi(\lambda x)$.
For fixed $\lambda \in (0,1)$, we consider $u^\lambda_n=\psi_\lambda u_n$. Then, letting $n\to\infty$, we have
$u^\lambda_n \rightharpoonup \psi_\lambda u$ in $\dot{H}^s\Om)$, because $\psi_\lambda$ is a multiplier on $\dot{H}^s(\Om)$, and 
$|u^\lambda_n|^{2^{\ast}}dx \tows \nu_\lambda=|\psi_\lambda|^{2^{\ast}}\nu$ in $\mathcal{M}(\mathds{R}^N)$.

If we assume that \eqref{quantnu} holds for each of these limiting measures $\nu_\lambda$ (possibly adding further Dirac masses in $B_{\lambda^{-1}}\cap\Om$ as $\lambda$ gets smaller), then the number of atoms of $\nu_\lambda$ is clearly uniformly bounded and for $0<\lambda<1$ and, for $0<\lambda<\lambda_0$ in the location is independent of $\lambda$ in $B_{\lambda_0^{-1}}$ (recall that there is a uniform bound in $\dot{H}^s$). Thus $\nu_\lambda \tows \nu$ as $\lambda \to 0$, hence  \eqref{quantnu} holds for $\nu$ as desired.

\vspace{2mm}

Let $\Omega \subset \mathds{R}^N$ be bounded and let us prove \eqref{quantnu} and \eqref{quantmu}.  Given $\varphi \in C^\infty_0(\mathds{R}^N)$, the Sobolev inequality \eqref{eq_sobolev} yields
\begin{equation}
\label{eq_sob1}
\left( \int_{\mathds{R}^N} |\varphi|^{2^{\ast}} |u_n|^{2^{\ast}}dx \right)^{\!\frac{2}{2^{\ast}}}\leq (S^{\ast})^{\frac{2}{2^{\ast}}} \|(-\Delta)^{\frac{s}{2}}(\varphi u_n)\|^{2}_{L^2(\mathds{R}^N)}
\end{equation}
and
\begin{equation}
\label{eq_sob2}
\left( \int_{\mathds{R}^N} |\varphi|^{2^{\ast}} |v_n|^{2^{\ast}} dx\right)^{\!\frac{2}{2^{\ast}}}\leq (S^{\ast})^{\frac{2}{2^{\ast}}} \|(-\Delta)^{\frac{s}{2}} (\varphi v_n)\|^{2}_{L^2(\mathds{R}^N)}.
\end{equation}
In view of Lemma \ref{lem_commutator} we have
$$
\|(-\Delta)^{\frac{s}{2}} (\varphi v_n)\|^{2}_{L^2(\mathds{R}^N)}= \|\varphi (-\Delta)^{\frac{s}{2}}  v_n \|^{2}_{L^2(\mathds{R}^N)}+o(1) \ \ \text{as}  \ n \to\infty.
$$
Passing to the limit in \eqref{eq_sob2} and using again Lemma \ref{lem_commutator} we get
\begin{equation}
\label{eq_sob3}
\left( \int_{\mathds{R}^N} |\varphi|^{2^{\ast}} d\hat{\nu}\right)^{\frac{2}{2^{\ast}}}=\ \lim_{n\to \infty} \left( \int_{\mathds{R}^N} |\varphi|^{2^{\ast}} |v_n|^{2^{\ast}}dx\right)^{\frac{2}{2^{\ast}}}\leq  \ (S^{\ast})^{\frac{2}{2^{\ast}}}\!\int_{\mathds{R}^N} \varphi^2 d\hat{\mu} \, , 
\end{equation}
i.~\!e. the measures $\hat{\nu}$ and $\hat{\mu}$ satisfy the reverse H\"older inequality \eqref{eq_reverse} with $p=2$, $r=2^{\ast}$ and $C=(S^{\ast})^{1/2^{\ast}}$. Thus, Lemma \ref{lem_reverse} gives the decomposition for $\hat{\nu}$ and in turn for $\nu= |u|^{2^{\ast}}dx+\hat{\nu}$, i.~\!e. \eqref{quantnu} holds. 

\vspace{1mm}

In order to prove \eqref{quantmu}, note that as $n \to \infty$ we have $v_n=u_n-u \rightharpoonup 0$ in $\dot{H}^s(\Om)$ (hence $(-\Delta)^{\frac{s}{2}}(u_n-u) \rightharpoonup 0$ in $L^2(\R^N)$), thus Lemma \ref{lem_commutator} gives
\begin{eqnarray}\label{eq_sob4} 
&& \!\!\!\!\int_{\R^N}| (-\Delta)^{\frac{s}{2}} (\varphi u_n)|^2 dx \nonumber \\
&& \ \,  =  \int_{\R^N}| (-\Delta)^{\frac{s}{2}}(\varphi u)|^2 dx+  \int_{\R^N}| \varphi (-\Delta)^{\frac{s}{2}} ( u_n-u)|^2 dx + o(1) \nonumber \\
\nonumber \\
&&  \ \, =  \, \int_{\R^N} | (-\Delta)^{\frac{s}{2}} (\varphi u)|^2 dx +\int_{\R^N} | \varphi (-\Delta)^{\frac{s}{2}} u_n|^2 dx -\int_{\R^N}| \varphi (-\Delta)^{\frac{s}{2}} u|^2 dx +o(1) \nonumber \\
\nonumber \\ 
&& \ \,  =  \, \int_{\R^N} | (-\Delta)^{\frac{s}{2}} (\varphi u)|^2 dx -\int_{\R^N}| \varphi (-\Delta)^{s/2}  u|^2 dx  +\int_{\R^N} | \varphi |^2 d\mu   +o(1) \, .
\end{eqnarray}
\vspace{1mm}

Combining \eqref{eq_sob1} and \eqref{eq_sob4}, as $n  \to \infty$ we obtain
\begin{eqnarray}
\label{eq_sob5}
\left( \int_{\R^N} |\varphi|^{2^{\ast}} d\nu \right)^{\!\frac{2}{2^{\ast}}} & \leq &  (S^{\ast})^{\!\frac{2}{2^{\ast}}}\Bigg(  \int_{\R^N} | (-\Delta)^{\frac{s}{2}} (\varphi u)|^2 dx \nonumber \\
&& \qquad \quad  \ -\, \int_{\R^N}| \varphi (-\Delta)^{\frac{s}{2}}  u|^2 dx   +\int_{\R^N}| \varphi |^2 d\mu \Bigg)
\end{eqnarray}
for any $\varphi \in C^\infty_0(\mathds{R}^N)$.

Since $\nu$ satisfies~\eqref{quantnu}, choosing $\varphi_{x_j,\lambda}(x)=\varphi(x_j+ \lambda^{-1}x))$ in~\eqref{eq_sob5} as a test function, Lemma~\ref{lem_cutoff} and dominated convergence  as $\lambda \to 0$ yield 
$$
\int_{\R^N} | (-\Delta)^{\frac{s}{2}} (\varphi_{x_j,\lambda} u)|^2 dx -\int_{\R^N}| \varphi_{x_j,\lambda} (-\Delta)^{\frac{s}{2}} u|^2 dx =o(1),
$$
whence $\nu\geq \sum_j\nu_j\delta_{x_j}$ implies
$\mu \geq \sum_j \mu_j \delta_{x_j}$ for some $\mu_j>0$ such that $\nu_j\leq S^{\ast}\mu_j^{\frac{2^\ast}{2}}$. 
\vspace{1mm}

Note that \eqref{quantmu} follows easily because $\sum_{j}\mu_j \delta_{x_j}$ and $|(-\Delta)^{\frac{s}{2}} u|^2dx$ are mutually singular, $\mu \geq  \sum_{j} \mu_j \delta_{x_j}$ and $\mu~\geq~|(-\Delta)^{\frac{s}{2}} u|^2dx$ (the latter inequality by weak lower semicontinuity in $L^2$), hence~\eqref{quantmu} holds. In order to conclude, it remains to observe that $\text{spt}\, \tilde{\mu}\subseteq \Omb$, i.~\!e., $\int\varphi^2d\mu = \int\varphi^2|(-\Delta)^{\frac{s}{2}}u|^2dx$ for any $\varphi\in C^{\infty}_0(\R^N\setminus \Omb)$, which is a straightforward consequence of equation~\eqref{eq_sob4} as $n\to \infty$.
\end{proof}

\begin{rem}
{\rm
In the proof of formula~\eqref{quantmu}, $\Om$ is assumed to be bounded just in order to apply Lemma~\ref{lem_commutator}. We don't know if the latter, and in turn formula~\eqref{quantmu},  holds also for unbounded domains.
}
\end{rem}

A simple consequence of the previous theorem is the following result, which will be useful in the next section and which shows that on bounded domains there is no energy loss in the concentration process. 

\begin{prop}
\label{pro_supporto}
Let $0<2s<N$, let  $\Omega \subset \mathds{R}^N$ be a bounded open set and let $\{u_n\}\subset \dot{H}^s(\Omega)$ such that $u_n \rightharpoonup 0$ as $n \to \infty$. For any open set $A\subseteq \mathds{R}^N$ such that $\overline{\Omega} \cap \overline A=\emptyset$ we have $\dys \int_A |(-\Delta)^{\frac{s}{2}} u_n|^2 dx \to 0$ as $n \to \infty$.
\end{prop}
\begin{proof}
In view of the energy concentration  described in formula~\eqref{quantmu} of Theorem~\ref{the_cca}, the conclusion clearly holds when $A$ is bounded, so it is enough to prove the claim when $A=\mathds{R}^N \setminus \overline{B}$ and $B\subset \mathds{R}^N$ is some Euclidean ball sufficiently large.

\vspace{1mm}

Let us choose $B$  such that $2\Omega \subset B$ and let $\varphi \in C^\infty_0(B)$ such that $\varphi \equiv 1$ on $\overline{B/2} \supset \Omega $.
Applying Lemma \ref{lem_commutator} we have
\begin{eqnarray*}
\int_A |(-\Delta)^{\frac{s}{2}} u_n|^2dx & \leq & \int_{\R^N} (1-\varphi)^2 |(-\Delta)^{\frac{s}{2}}u_n|^2dx \\
\\
&   = & \ \int_{\R^N} \left| \left[ 1-\varphi, (-\Delta)^{\frac{s}{2}} \right] u_n \right|^2dx \\
& = & \int_{\R^N} \left| \left[ \varphi, (-\Delta)^{\frac{s}{2}} \right] u_n \right|^2dx  \ \, \overset{n\to \infty}{\longrightarrow} \ \, 0
\end{eqnarray*}
and the proof is complete.
\end{proof}

\subsection{Asymptotic behaviour of optimal sequences}

In case of bounded domains the situation simplifies considerably and we have the following result, that is
a direct consequence of Theorem~\ref{the_cca}  
 and describes the behavior of optimal sequences for the variational problem~\eqref{pbsobolev2} in bounded domains.
\begin{corollary}\label{cor_concentrazio-sob}
Let $\Om\subset \mathds{R}^N$ a bounded open set and let $\{u_n\} \subset \dot{H}^s(\Om)$ be a maximizing sequence for the critical Sobolev inequality {\rm\eqref{pbsobolev2}}. Then, up to subsequences, $\{u_n\}$ concentrates at one point $x_0 \in \Omb$ in the sense that $|u_n|^{2^{\ast}}dx \tows S^{\ast} \delta_{x_0}$ and  $|(-\Delta)^{\frac{s}{2}}u_n|^{2}dx \tows \delta_{x_{0}}$ in $\mathcal{M}(\mathds{R}^N)$.
\end{corollary}
\begin{proof}
The result easily follows from the concentration-compactness alternative in Theorem \ref{the_cca}. One of the keypoint in the proof is the well-known convexity trick by Lions. 

Let $\{u_n\} \subset \dot{H}^s(\Omega)$ be a maximizing sequence for the critical Sobolev inequality \eqref{pbsobolev2}. Then, up to subsequences, $u_n\tow u$ in $\dot{H}^s(\Omega)$, $\dys \int_{\Om} |u_n|^{2^\ast}dx\to S^{\ast}$ and also $\dys |u_n|^{2^{\ast}}\!dx \tows \nu \in \mea$ with $\nu(\overline{\Om})=S^{\ast}$.\vspace{1mm}

By formula \eqref{quantnu} in Theorem \ref{the_cca}, we have
\begin{equation}\label{eq_ccadis}
\dys S^{\ast}   =    \nu(\overline{\Om}) = \int_\Om |u|^{2^{\ast}}\!dx +\sum_{j} \nu_j.
\end{equation}

Combining the Sobolev inequality \eqref{eq_sobolev} with \eqref{quantnu}-\eqref{quantmu}, we get
\begin{equation}\label{eq_sobdis}
\dys \int_\Om |u|^{2^{\ast}}\!dx +\sum_{j} \nu_j \ \leq\  S^{\ast}\left(\int_{\R^N} |(-\Delta)^{\frac{s}{2}}u|^2dx\right)^{\!\!\frac{2^{\ast}}{2}}\!+S^{\ast}\sum_{j}\mu_j^{\frac{2^{\ast}}{2}},
\end{equation}
where $\mu_j$ are atomic coefficients of the measure $\mu\in \mea$, that is the limit in the sense of measures of the sequence $|(-\Delta)^{\frac{s}{2}} u_n|^2dx$.

\vspace{1mm}

Taking formula \eqref{quantmu} and Proposition \ref{pro_supporto} into account we have
\begin{eqnarray}\label{eq_convdis}
S^{\ast}\left(\int_{\R^N} |(-\Delta)^{\frac{s}{2}}u|^2dx\right)^{\!\!\frac{2^{\ast}}{2}}+S^{\ast}\sum_{j}\mu_j^{\frac{2^{\ast}}{2}}
& \leq & S^{\ast}\left(\int_{\R^N} |(-\Delta)^{\frac{s}{2}}u|^2dx+\sum_{j} \mu_j\right)^{\!\!\frac{2^{\ast}}{2}} \nonumber \\
& \leq & S^{\ast}\mu(\R^N) \ = \ S^{\ast}, 
\end{eqnarray}
because $\| u_n\|_{\dot{H}^s}=1$ for each $n$ and, in view of Proposition \ref{pro_supporto}, there is no loss of energy in the limit.

\vspace{1mm}

Therefore, combining \eqref{eq_ccadis}, \eqref{eq_sobdis} and \eqref{eq_convdis}, we see that all the inequalities must be equalities. Since the Sobolev constant is not attained on bounded domains and the function $t\mapsto t^{\frac{2^{\ast}}{2}}$ is strictly convex, it follows that $\tilde{\mu}=0$, $u$ is zero and only one of the $\mu_j$'s and $\nu_j$'s can be nonzero in \eqref{quantnu}-\eqref{quantmu}. Hence, concentration occurs at one point $x_0\in \Omb$ as claimed.
\end{proof}

\vspace{1mm}

We conclude this section with the asymptotic analysis of the maximizers for the variational problem in~\eqref{problemaeps}, proving the claims stated in Theorem~\ref{the_concentrazione1}.

\begin{proof}[\bf Proof of Theorem \ref{the_concentrazione1}]

First, we claim that
\begin{equation}\label{eq_3astast}
\dys \limsup_{\eps\to0}S^*_\eps \leq S^*.
\end{equation}
Indeed, taking $\ue \in \dot{H}^s(\Om)$ a maximizer for $S^*_\eps$,  by H\"older inequality
we have
\begin{eqnarray*}
\dys S^*_\eps \, = \, F_\eps(\ue)  \, =  \, \int_\Om |\ue|^{2^*-\eps}dx \, \leq \, \left(\int_\Om|\ue|^{2^*}\right)^{\!\!\frac{2^*-\eps}{2^*}}\!|\Om|^{\frac{\eps}{2^*}}
 \leq \, (S^*)^{\!\frac{2^*-\eps}{2^*}}|\Om|^{\frac{\eps}{2^*}}.
\end{eqnarray*}
Thus, inequality~\eqref{eq_3astast} follows as $\eps \to 0$.
\vspace{2mm}

The reverse inequality easily follows from the pointwise convergence of $F_\eps$ to $F_\Om$ with a standard argument. Indeed, for every $\delta>0$ there exists $u_\delta\in \dot{H}^s(\Om)$ such that $\|u_\delta\|_{\dot{H}^s}\leq 1$ and
\begin{equation}\label{eq_4ast}
\dys F_\Om(u_\delta)>S^*-\delta.
\end{equation}
Clearly, for such function $u_\delta$, we have $S^*_\eps\geq 
F_\eps(u_\delta)$.
Thus, combining the previous inequality with~\eqref{eq_4ast} and passing to the limit as $\eps$ goes to zero, we get
\begin{eqnarray*}
\dys \liminf_{\eps\to 0}S^*_\eps \, \geq \, \lim_{\eps\to 0}F_\eps(u_\delta) \, = \,  F_\Om(u_\delta) \, \geq \, S^* -\delta
\end{eqnarray*}
and claim~(i) follows as $\delta\to 0$ in view of~\eqref{eq_3astast}.

\vspace{2mm}

The concentration result~(ii) for the sequence $\{\ue\}$ of maximizers of $S^*_\eps$ now is straightforward. Due to~(i) the sequence $\ue$ is a maximizing sequence for $F_\Om$, hence, Corollary~\ref{cor_concentrazio-sob} ensures that, up to subsequences, $\{\ue\}$ concentrate at one point $x_0\in\Omb$, in the sense that
$|\ue|^{2^\ast}dx\tows S^{\ast}\delta_{x_{0}}$ and
$|(-\Delta)^{\frac{s}{2}} \ue|^2dx \tows \delta_{x_{0}}$ in $\mathcal{M}(\R^N)$.
\vspace{2mm}

Finally, since $\{ u_\varepsilon\}$ is also an optimal sequence for the Sobolev inequality \eqref{eq_sobolev}, by Theorem \ref{the_optimizseq} we deduce the convergence statement in~(iii) under suitable rescaling.

Note that, following the proof of Theorem \ref{the_optimizseq} and taking claim~(ii) into account, the scaling parameters there clearly satisfy $x_n\to x_0$ and $\lambda_n \to 0$ as $n\to \infty$ and the proof of~(iii) is complete.
\end{proof}

\vspace{2mm}

\noindent
\\ {\bf Acknowledgments.} 
We are  indebted
 with Luis Vega for useful discussions about Sobolev inequalities and weighted estimates for the Riesz potentials, and for having drawn our attention on \cite{KPV}. We would like to thank Piero D'Ancona for having pointed out to us the paper \cite{SST11}.
\\[0.5ex] The first author has been supported by the \href{http://prmat.math.unipr.it/~rivista/eventi/2010/ERC-VP/}{ERC grant 207573 ``Vectorial Problems''}.

\vspace{3mm}

\end{document}